%% file: main.tex
\documentclass[12pt,leqno]{amsart}
\usepackage[utf8]{inputenc}

\title[On the computation of the vanishing locus of a finitely presented functor]{On the computation of the vanishing locus of a finitely presented functor}
\author{Sebastian Posur}

\usepackage{csquotes}

\input{header.tex}


\begin{document}

\begin{abstract}
We discuss invariants which are helpful for the computation of the vanishing locus of a finitely presented functor $\mathcal{G}$, i.e., the set of points in the Ziegler spectrum on which $\mathcal{G}$ vanishes. These invariants are: the rank of $\mathcal{G}$, the supports of its co- and contravariant defect, and the class of $\mathcal{G}$ in the Grothendieck group of the category of finitely presented functors. We show that these invariants determine the vanishing locus in the case of a finitely presented functor over a Dedekind domain.
\end{abstract}


\maketitle
\setcounter{tocdepth}{4}
\tableofcontents

\section{Introduction}
The \emph{functorial approach} is the idea that some mathematical concepts gain clarity by shifting our attention from the study of modules to the study of functors.
Auslander demonstrated the strength of such an approach in the realm of representation theory \cite{AuslanderFunctorial}.
This paper is motivated by the observation that a functorial approach is also beneficial for the behavioral approach to algebraic systems theory initiated by Willems \cite{WilParadigms}.

The category $R\modl\modl$ of finitely presented functors on finitely presented modules over a ring $R$ plays the following important role in the behavioral approach to algebraic systems theory: its Serre quotients can be regarded as ambient categories for behaviors, i.e., as a mathematical context for the main objects of interest in the behavioral approach to algebraic systems theory. This point of view is elaborated in \cite{PosBehI}. 

Every object $\mathcal{G}$ of the category $R\modl\modl$ is given as the cokernel of a natural transformation between hom functors, i.e., by an exact sequence of functors
\begin{center}
   \begin{tikzpicture}[mylabel/.style={fill=white},baseline=(A)]
        \coordinate (r) at (4.5,0);
        \node (A) {$\Hom( N', - )$};
        \node (B) at ($(A)+(r)$) {$\Hom( N, - )$};
        \node (C) at ($(B) + 0.5*(r)$) {$\mathcal{G}$};
        \node (D) at ($(C) + 0.25*(r)$) {$0$,};
        \draw[->,thick] (A) to node[above]{$\Hom( \alpha, - )$} (B);
        \draw[->,thick] (B) to (C);
        \draw[->,thick] (C) to (D);
  \end{tikzpicture}
\end{center}
where $\alpha: N \rightarrow N'$ is a morphism in the category $R\modl$ of finitely presented modules over $R$. Evaluation at an arbitrary $R$-module $M$ defines an exact functor
\[
\ev{M}: R\modl\modl \rightarrow \Ab
\]
whose target is the category of all abelian groups.
Now, the abelian ambient category for behaviors over $M$ is defined as the following Serre quotient:
\[
\ab{M} := \frac{R\modl\modl}{\kernel(\ev{\M})}.
\]
Notions like the controllability or the observability of a behavior can be studied within the context provided by the category $\ab{M}$, see \cite[Section 8]{PosBehI}.

The category $R\modl\modl$ is amenable to explicit computations whenever the
additive closure of $R$ has decidable homotopy equations, see \cite{PosFree}.
Since algebraic systems theory depends on effective computations, it is a natural goal to find situations in which we can also render Serre quotients of $R\modl\modl$ constructive. 
For a Serre subcategory $\CC \subseteq \AC$ of an abelian category, the membership problem is the decision problem of whether a given object $A \in \AC$ lies in $\CC$ or not.
If $\AC$ is a computable abelian category (in the sense of \cite{PosCCT}), then the Serre quotient $\AC/\CC$ is computable if the membership problem is decidable.
Thus, it is natural to ask when the membership problem for Serre subcategories of $R\modl\modl$ is decidable.

From a theoretical point of view, Serre subcategories of $R\modl\modl$ are classified by the closed sets of the (left or right) Ziegler spectrum $\Zg(R)$ of $R$. The points in $\Zg(R)$ are (isomorphism classes of) indecomposable pure-injective modules. The classification works as follows: given a closed set $X \subseteq \Zg(R)$, all finitely presented functors which vanish on $X$ give a Serre subcategory. Conversely, given a Serre subcategory $\CC \subseteq R\modl\modl$, its vanishing locus 
\[
\V( \CC ) := \{ P \in \Zg(R) \mid \text{$\mathcal{G}(P) \cong 0$ for all $\mathcal{G} \in \CC$} \}
\]
defines a closed subset of $\Zg(R)$. It follows that whenever we know the 
corresponding closed subset $X \subseteq \Zg(R)$ of a Serre subcategory, the membership problem means testing whether the vanishing locus $\V( \mathcal{G} )$ of
a given finitely presented functor $\mathcal{G}$ contains $X$.

This paper focuses on computational tools which are helpful for the determination of (parts of) the vanishing locus of a finitely presented functor.
In \Cref{section:fpfunctors}, we recall the theory of finitely presented functors needed for our purposes. 
In \Cref{section:serrequots_of_free_abelian_cats}, we introduce the category $\ab{M}$ and show how to identify this category in some special cases.
In \Cref{section:ziegler_spectrum}, we recall some theoretical aspects of the Ziegler spectrum, where we use the language of frames for our description.

In \Cref{section:vanishing_locus_kernel_defect}, we describe the following two parts of the vanishing locus of a finitely presented functor $\mathcal{G}$ over a commutative noetherian ring:
\[
\V( \mathcal{G} ) \cap \V( \kernel( \defect ) ) \text{\hspace{2em}and\hspace{2em}} \V( \mathcal{G} ) \cap \V( \kernel( \codefect ) ),
\]
where $\defect$ and $\codefect$ are canonical exact functors of types
\[
R\modl\modl \xrightarrow{\defect} (R\modl)^{\op}  \text{\hspace{2em}and\hspace{2em}} R\modl\modl \xrightarrow{\codefect} (\modr R),
\]
respectively.
In particular, we show that these parts can be identified with the complement of the supports of the modules $\defect( \mathcal{G} ) \in R\modl$ and $\codefect( \mathcal{G} ) \in \modr R$, respectively, and are thus amenable to tools of computer algebra, see \Cref{theorem:checking_kernels_of_defects} and \Cref{remark:checking_kernels_of_defects}.

In \Cref{section:grothendieck_group_of_free_abelian_cat}, we describe the Grothendieck group $K_0(R\modl\modl)$ for $R$ an arbitrary ring.
It directly follows from Grothendieck's resolution theorem and the description of projectives and injectives in $R\modl\modl$ that we have isomorphisms
\[
K_0( R\modl\modl ) \cong K_0( |R\modl| )  \text{\hspace{2em}and\hspace{2em}} K_0( R\modl\modl ) \cong K_0( |\modr R| ),
\]
where $K_0( |R\modl| )$ and $K_0( |\modr R| )$ are the Grothendieck groups of the categories of left/right finitely presented $R$-modules merely regarded as additive categories, see \Cref{theorem:K_0_fpfunctors}. In a short digression, we explain the resulting isomorphism $K_0( |R\modl| ) \cong K_0( |\modr R| )$ as the decategorification of a triangulated equivalence
\[
\Kb( R \modl )^{\op} \simeq \Kb( \modr R )
\]
that can be constructed using the Auslander transpose (\Cref{lemma:triangulated}).

In \Cref{section:additive_invariants_for_the_detection}, we ask which additive invariants of $R\modl\modl$ are helpful for the membership problem for Serre subcategories. We say that an additive invariant $\alpha$ detects being zero at an arbitrary $R$-module $M$ if
\[
\text{$\alpha( \mathcal{G} ) = 0$ if and only if $\mathcal{G}(M) \cong 0$}
\]
for all $\mathcal{G} \in R\modl\modl$.
We prove that there are such invariants (with values in $\Z$) for $M$ equal to a right artinian ring $S$ regarded as an $R$-module by restriction along a ring map $R \rightarrow S$, see \Cref{theorem:K_0_zero_detection}.
Later, we give an example of pure-injective modules $P$ and $A$ such that there are no additive invariants that detect being zero either at $P$ or $A$, but such that there is such an additive invariant for their direct sum $P \oplus A$, see \Cref{example:no_additive_invariant_that_detects_zero} and \Cref{corollary:AplusP_detection}.

In \Cref{section:the_rank}, we define and discuss the rank of a finitely presented functor over an integral domain. In \Cref{section:applications}, we apply our findings to the case of $R$ being a Dedekind domain. In this case, all points of the Ziegler spectrum are explicitly known. 
The Grothendieck group can be explicitly described in terms of the invariants provided by the structure theorem for finitely generated modules over Dedekind domains, see \Cref{corollary:K_0_Dedekind}.
We describe all categories of the form $\ab{M}$ for $M$ being a pure-injective module in \Cref{theorem:abM_for_pure_injectives_Dedekind}.
We show that the vanishing locus of an object $\mathcal{G}$ in $R\modl\modl$ can be computed as the following union (see \Cref{theorem:main_Dedekind}):
\[
\V( \mathcal{G} ) = \V_{\Hilb}( \mathcal{G} ) \cup \V_{\defect}( \mathcal{G} ) \cup \V_{\codefect}( \mathcal{G} ) \cup \V_{\rank}( \mathcal{G} )
\]
where the sets on the right hand side are computable by the following four invariants: the Hilbert function (introduced in \Cref{subsection:hilbert_function}), the contravariant defect, the covariant defect, and the rank of $\mathcal{G}$.

\section{Notation and conventions}

\begin{notation}\label{notation:cats}
Categories are printed in bold. For a category $\CC$, we denote its class of objects by $\obj( \CC )$.
If $\AC$ is an abelian category, then we denote
\begin{itemize}
    \item its underlying additive category by $|\AC|$,
    \item the full subcategory spanned by all projective objects by $\Proj( \AC )$, which we regard as an additive category,
    \item the full subcategory spanned by all injectives objects by $\Inj( \AC )$, which we regard as an additive category.
\end{itemize}
We denote by $\Ab$ the category of abelian groups.
For $R$ a ring, we make use of the following categories:
\begin{itemize}
    \item $R\Modl$: the category of all left $R$-modules.
    \item $\Modr R$: the category of all right $R$-modules.
    \item $R\modl$: the category of finitely presented left $R$-modules.
    \item $\modr R$: the category of finitely presented right $R$-modules.
\end{itemize}
If $Q$ is a field, we denote by $Q\vecl$ the category of finite dimensional $Q$-vector spaces.
Moreover, we will make use of the fact that right modules over $R$ can be regarded as left modules over $R^{\op}$ and vice versa.
\end{notation}

\begin{para}
A ring $R$ is \textbf{left coherent} if every finitely generated left submodule of ${_R}R$ is finitely presented.
This is equivalent to $R\modl$ being an abelian category.
Dually, $R$ is called \textbf{right coherent} if $R^{\op}$ is left coherent.
\end{para}

\begin{notation}
For a morphism of rings $R \xrightarrow{\ringmap} S$ and $M \in S\Modl$, we denote the restriction of scalars by $M_{|\ringmap} \in R\Modl$.
\end{notation}

\begin{notation}
For a module $M \in R\Modl$, we denote its injective hull by $E( M )$.
\end{notation}

\begin{convention}
All functors between additive categories are assumed to be additive.
\end{convention}

\begin{notation}
For a commutative ring $R$, we refer to its set of \emph{prime} ideals by $\Spec(R)$ and its set of \emph{maximal} ideals by $\mSpec(R)$.
If $\pid \subseteq R$ is a prime ideal, we denote the localization of $R$ at $\pid$ by $R_{\pid}$.
\end{notation}

\begin{para}
A \textbf{principal ideal domain} (PID) is a commutative ring $R$ with the following properties: it is a domain, i.e., it has no zero-divisors, its Krull dimension is $1$, and every ideal of $R$ is principal.
Note that we exclude fields in our definition of principal ideal domains.
A \textbf{discrete valuation ring} (DVR) is a local PID.
A \textbf{Dedekind domain} is a commutative noetherian domain $R$ of Krull dimension $1$ such that the localization $R_{\pid}$ is a DVR for every $\pid \in \mSpec(R)$. Note that we include commutativity in our definition of a Dedekind domain.
\end{para}

\section{Preliminaries on finitely presented functors}\label{section:fpfunctors}

We recall some aspects of the theory of finitely presented functors on finitely presented modules. The notion of a finitely presented functor goes back to Auslander \cite{AuslanderCoherent}.

\begin{definition}\label{definition:fp_func}
Let $\AC$ be an additive category. 
A functor $\mathcal{G}$ of type $\AC \longrightarrow \Ab$
is called \textbf{finitely presented} if there exist objects $A, B \in \AC$, a morphism $A \xrightarrow{\alpha} B$, and an exact sequence of functors (called a \textbf{presentation of $\mathcal{G}$ given by $\alpha$}):
\begin{center}
   \begin{tikzpicture}[mylabel/.style={fill=white},baseline=(A)]
        \coordinate (r) at (4.5,0);
        \node (A) {$\Hom( B, - )$};
        \node (B) at ($(A)+(r)$) {$\Hom( A, - )$};
        \node (C) at ($(B) + 0.5*(r)$) {$\mathcal{G}$};
        \node (D) at ($(C) + 0.25*(r)$) {$0$.};
        \draw[->,thick] (A) to node[above]{$\Hom( \alpha, - )$} (B);
        \draw[->,thick] (B) to (C);
        \draw[->,thick] (C) to (D);
  \end{tikzpicture}
\end{center}
A morphism of finitely presented functors is a natural transformation.
We denote the resulting category by $\AC\modl$.
\end{definition}

\begin{para}\label{para:up_of_Rmodmod}
Let $R$ be a ring.
When we set $\AC := R\modl$ within \Cref{definition:fp_func}, we get the category $R\modl\modl$.
It is an abelian category and a model of the free abelian category of $R$. More precisely, we have a canonical functor
\[
R \rightarrow R\modl\modl
\]
where we regard $R$ as a category with a single object and with morphisms given by the ring elements of $R$.
This canonical functor sends the single object of $R$ to the \textbf{forgetful functor} $\mathcal{F} := \Hom(R,-): R\modl \rightarrow \Ab$. Now, the universal property of $R\modl\modl$ states that we have an equivalence of categories
\[
\Hom_{\text{exact}}( R\modl\modl, \AC ) \xrightarrow{\simeq} \Hom_{\text{additive}}( R, |\AC| )
\]
given by restricting an exact functor of type $R\modl\modl \rightarrow \AC$, where $\AC$ is an abelian category, along the canonical functor $R \rightarrow R\modl\modl$ in order to get an additive functor of type $R \rightarrow |\AC|$.
\end{para}

\begin{remark}
The category $R\modl\modl$ is skeletally small, i.e., its isomorphism classes form a set. Thus, up to equivalence, we can treat $R\modl\modl$ as a small category, i.e, as a category whose objects form a set.
\end{remark}

\begin{remark}\label{remark:ev_at_RMod}
Moreover, we can regard $R\modl\modl$ as a full subcategory of $R\Modl\modl$.
In particular, we can evaluate $\mathcal{G} \in R\modl\modl$ at an arbitrary $R$-module (see \cite[Paragraph 4.14]{PosBehI} for details).
\end{remark}

We can describe all projectives and injectives in $R\modl\modl$:

\begin{lemma}\label{lemma:projectives_injectives_in_Rmodmod}
We have equivalences of categories
\begin{align*}
(R\modl)^{\op} &\xrightarrow{\simeq} \Proj( R\modl\modl ) \\
M &\mapsto \Hom(M,-)
\end{align*}
and
\begin{align*}
\modr R &\xrightarrow{\simeq} \Inj( R\modl\modl ) \\
N &\mapsto (N \otimes -)
\end{align*}
\end{lemma}
\begin{proof}
All hom functors are projective by the Yoneda lemma.
Moreover, for $\mathcal{G} \in R\modl\modl$, there is an epimorphism from a hom functor onto $\mathcal{G}$. Thus, if $\mathcal{G}$ is projective, it is a direct summand of a hom functor. Since $(R\modl)^{\op}$ is idempotent complete, $\mathcal{G}$ is a hom functor.
The statement for injectives follows from the Auslander-Gruson-Jensen duality, see \Cref{theorem:agj_duality}.
\end{proof}

\begin{theorem}[Auslander-Gruson-Jensen duality]\label{theorem:agj_duality}
We have an equivalence
\[
\DAGJ: (R\modl\modl)^{\op} \rightarrow R^{\op}\modl\modl
\]
with $\DAGJ( \Hom( M, - ) ) \cong (- \otimes M)$ for $M \in R\modl$.
\end{theorem}
\begin{proof}
See, e.g., \cite[Theorem 5.13]{PosBehI}.
\end{proof}

\begin{definition}\label{definition:copresentations}
An exact sequence in $R\modl\modl$ of the form
\begin{center}
   \begin{tikzpicture}[mylabel/.style={fill=white}]
    \coordinate (r) at (4,0);
    \node (A) {$0$};
    \node (B) at ($(A)+0.5*(r)$) {$\mathcal{G}$};
    \node (C) at ($(B) + 0.5*(r)$) {$(N \otimes -)$};
    \node (D) at ($(C) + (r)$) {$(N' \otimes -)$};
    \draw[->,thick] (A) to (B);
    \draw[->,thick] (B) to (C);
    \draw[->,thick] (C) to node[above]{$(\alpha \otimes -)$}(D);
    \end{tikzpicture}
\end{center}
is called a \textbf{copresentation} of $\mathcal{G} \in R\modl\modl$ given by the morphism
$\alpha: N \rightarrow N'$ in $\modr R$.
By \Cref{theorem:agj_duality}, every object in $R\modl\modl$ has such a copresentation.
\end{definition}

\begin{lemma}\label{lemma:projectives_length2}
If $\mathcal{G} \in R\modl\modl$ has a presentation given by $\alpha: M \rightarrow N$, then we get a projective resolution of $\mathcal{G}$ of length $2$ as follows:
\[
   \begin{tikzpicture}[mylabel/.style={fill=white},baseline=(A)]
        \coordinate (r) at (4.5,0);
        \node (Za) {$0$};
        \node (Z) at ($(Za)+0.5*(r)$){$\Hom( \cokernel( \alpha ), - )$};
        \node (A) at ($(Z)+(r)$) {$\Hom( N, - )$};
        \node (B) at ($(A)+(r)$) {$\Hom( M, - )$};
        \node (C) at ($(B) + 0.5*(r)$) {$\mathcal{G}$};
        \node (D) at ($(C) + 0.25*(r)$) {$0$};
        \draw[->,thick] (Za) to (Z);
        \draw[->,thick] (Z) to node[above]{$\Hom( \epsilon, - )$} (A);
        \draw[->,thick] (A) to node[above]{$\Hom( \alpha, - )$} (B);
        \draw[->,thick] (B) to (C);
        \draw[->,thick] (C) to (D);
  \end{tikzpicture}
\]
where $\epsilon: N \rightarrow \cokernel( \alpha )$ denotes the cokernel projection.
\end{lemma}
\begin{proof}
This is a direct computation.
\end{proof}

\begin{para}\label{para:bijective_objects}
Let $R$ be a ring and let $P$ be a finitely presented projective left $R$-module. Then $P^{\ast} := \Hom_R( P, R )$ is a finitely presented projective right $R$-module and we have an isomorphism
\[
\Hom(P,-) \cong (P^{\ast} \otimes -) \in R\modl\modl.
\]
In particular, these functors are both projective and injective objects in $R\modl\modl$.
\end{para}

\begin{para}\label{para:Auslander_transpose}
For $M \in R\modl$, let
\begin{center}
  \begin{tikzpicture}[mylabel/.style={fill=white}]
      \coordinate (r) at (3,0);
      
      \node (A) {$Q$};
      \node (B) at ($(A)+(r)$) {$P$};
      \node (C) at ($(B) + (r)$) {$M$};
      \node (D) at ($(C) + (r)$) {$0$};
      
      \draw[->,thick] (A) to (B);
      \draw[->,thick] (B) to (C);
      \draw[->,thick] (C) to (D);
  \end{tikzpicture}
\end{center}
be an exact sequence with $P,Q$ finitely presented projective $R$-modules.
Then the \textbf{Auslander transpose} of $M$ is the finitely presented right $R$-module $\Tr( M )$ defined by the following exact sequence in $\modr R$:
\begin{center}
  \begin{tikzpicture}[mylabel/.style={fill=white}]
      \coordinate (r) at (3,0);
      
      \node (A) {$0$};
      \node (B) at ($(A)+(r)$) {$\Tr(M)$};
      \node (C) at ($(B) + (r)$) {$Q^{\ast}$};
      \node (D) at ($(C) + (r)$) {$P^{\ast}$.};
      
      \draw[->,thick] (D) to (C);
      \draw[->,thick] (C) to (B);
      \draw[->,thick] (B) to (A);
  \end{tikzpicture}
\end{center}
The Auslander transpose plays an important role in Auslander-Reiten theory.
Moreover, it clearly depends on the presentation of the left module, and the usual approach of making the Auslander transpose functorial is by passing to stable module categories (see, e.g., \cite[Chapter 4]{ARS97} or \cite[Chapter 6]{Facchini}).
In \Cref{lemma:triangulated}, we will show how the Auslander transpose can be used to construct a triangulated equivalence $\Kb( R \modl )^{\op} \simeq \Kb( \modr R )$.
\end{para}

The following lemma simultaneously describes a projective resolution of a tensor functor and an injective resolution of a hom functor.

\begin{lemma}\label{lemma:proj_res_of_tensor_functor}
In the context of Paragraph \Cref{para:Auslander_transpose}, the following sequence is exact:
\begin{center}
  \begin{tikzpicture}[mylabel/.style={fill=white}]
      \coordinate (r) at (3.5,0);
      \coordinate (d) at (0,-1);
      
      \node (A) {$0$};
      \node (B) at ($(A)+0.75*(r)$) [scale=1]{$\Hom(M,-)$};
      \node (C) at ($(B) + (r)$) [scale=1]{$\Hom(P,-)$};
      \node (Cp) at ($(C) + (d)$) [scale=1]{$(P^{\ast} \otimes -)$};
      \node (D) at ($(C) + 0.85*(r)$) [scale=1]{$\Hom(Q,-)$};
      \node (Dp) at ($(D) + (d)$) [scale=1]{$(Q^{\ast} \otimes -)$};
      \node (E) at ($(D) + (r)$) [scale=1]{$(\Tr(M) \otimes -)$};
      \node (F) at ($(E) + 0.75*(r)$) {$0$};
      
      \draw[->,thick] (A) to (B);
      \draw[->,thick] (B) to (C);
      \draw[->,thick] (C) to (D);
      \draw[->,thick] (D) to (E);
      \draw[->,thick] (E) to (F);
      \draw[draw=none] (C) to node[rotate=90]{$\cong$} (Cp);
      \draw[draw=none] (D) to node[rotate=90]{$\cong$} (Dp);
  \end{tikzpicture}
\end{center}
\end{lemma}
\begin{proof}
This is a direct computation.
\end{proof}

\section{The category \texorpdfstring{$\ab{M}$}{ab(M)}}\label{section:serrequots_of_free_abelian_cats}

In this section, we turn our attention to Serre quotients of $R\modl\modl$.

\begin{para}\label{para:Serre_subcategories}
Let $\AC$ be an abelian category.
An additive full subcategory $\CC \subseteq \AC$ is called a \textbf{Serre subcategory} if it is closed under passing to subobjects, quotient objects, and extensions.
Given a class $T \subseteq \obj( \AC )$ of objects, we denote by $\langle T \rangle \subseteq \AC$ the Serre subcategory generated by $T$, i.e., the smallest Serre subcategory which contains $T$.
More concretely, an object $A \in \AC$ lies in $\langle T \rangle$ if there exists a chain of subobjects
\[
A = A_0 \supseteq A_1 \supseteq A_2 \supseteq \dots \supseteq A_{n-1} \supseteq A_n = 0
\]
for $n \geq 0$ such that each factor $A_{i}/A_{i+1}$ is isomorphic to a subfactor of an object in $T$ (see, e.g., \cite[Proposition 2.4.(6)]{Kanda}
\end{para}

\begin{para}
Let $\AC$ be an abelian category and let $\CC \subseteq \AC$ be a Serre subcategory.
The \textbf{Serre quotient} $\AC/\CC$ is an abelian category which comes equipped with a canonical exact quotient functor $\AC \rightarrow \AC/\CC$ whose \textbf{kernel}, i.e., the full subcategory spanned by all objects which are mapped to zero, is given by $\CC$.
It is universal (initial) among all exact functors whose kernel contains $\CC$.
For more details, see \cite[Section 6.1]{PosBehI}.
\end{para}

\begin{para}\label{para:def_ab}
Let $R$ be a ring and let $M \in R\Modl$.
By \Cref{remark:ev_at_RMod}, we can evaluate finitely presented functors at $M$, which gives us an exact functor
\begin{align*}
\ev{\M}: R\modl\modl &\rightarrow \Ab \\
\mathcal{G} &\mapsto \mathcal{G}(\M).
\end{align*}
We define the following Serre quotient category:
\[
\ab{M} := \frac{R\modl\modl}{\kernel(\ev{\M})}.
\]
This category can be regarded as an abelian ambient category for behaviors in algebraic systems theory (see \cite{PosBehI} for details).
We use the following notation for the canonical quotient functor on objects:
\begin{align*}
R\modl\modl &\rightarrow \ab{M} \\
\mathcal{G} &\mapsto  \asBeh{\mathcal{G}}{\M} 
\end{align*}
and note that
    \[
    \asBeh{\mathcal{G}}{\M} \cong 0 \text{\hspace{2em}$\Longleftrightarrow$ \hspace{2em}} \mathcal{G}(\M) \cong 0.
    \]
\end{para}

We recall several useful cases in which we can describe $\ab{M}$ more explicitly.

\begin{para}\label{para:defect}
Let $R$ be a left coherent ring.
Let $R \rightarrow (R\modl)^{\op}$ be the functor that maps $r \in R$ to the endomorphism $(x \mapsto xr) \in \Hom_R( R, R )$.
Then the functor induced by the universal property of the free abelian category
\[
R\modl\modl \xrightarrow{\defect} (R\modl)^{\op}
\]
is called the \textbf{(contravariant) defect}. It is an exact functor which sends $\Hom(N,-)$ to $N$ for every $N \in R\modl$ (see \cite[Example 5.9]{PosBehI}).
\end{para}

\begin{para}\label{para:defect_equiv}
A module $M \in R\Modl$ is \textbf{fp-injective} if the functor $\Ext^1( -, M )$ vanishes on all finitely presented modules. Note that every injective module is in particular fp-injective.
Moreover, $M$ is an \textbf{fp-cogenerator} if $\Hom( -, M )$ is faithful on the category of all finitely presented functors. Note that every cogenerator is in particular an fp-cogenerator.
These notions are relevant for the following situation (see \cite[Corollary 6.29]{PosBehI}): let $R$ be a left coherent ring and let $M \in R\Modl$ be fp-injective and an fp-cogenerator.
Then we get a commutative (up to natural isomorphism) diagram of exact functors
\begin{center}
  \begin{tikzpicture}[mylabel/.style={fill=white}]
    \coordinate (r) at (4,0);
    \coordinate (d) at (0,-2);
    \node (A) {$R\modl\modl$};
    \node (B) at ($(A) + 2*(r)$) {$(R\modl)^{\op}$};
    \node (C) at ($(A) + (d) + (r)$) {$\ab{M}$};
    \draw[->] (A) to node[above]{$\defect$} (B);
    \draw[->] (A) to (C);
    \draw[->] (C) to node[right, yshift=-0.5em]{$\simeq$}(B);
  \end{tikzpicture} 
\end{center}
where the unlabeled arrow is the canonical quotient functor and where the functor induced by the universal property of the Serre quotient is an equivalence.
\end{para}

\begin{para}\label{para:codefect}
Let $R$ be a right coherent ring.
Let $R \rightarrow \modr R$ be the functor that maps $r \in R$ to the endomorphism $(x \mapsto rx) \in \Hom_R( R, R )$.
Then the functor induced by the universal property of the free abelian category
\[
R\modl\modl \xrightarrow{\codefect} (\modr R)
\]
is called the \textbf{(covariant) defect}. It is an exact functor which sends $(N \otimes_R -)$ to $N$ for every $N \in \modr R$ (see \cite[Example 5.19]{PosBehI}).
\end{para}

\begin{para}\label{para:codefect_equiv}
A module $M \in R\Modl$ is \textbf{fp-faithfully flat} if it is flat and if $(- \otimes M)$ is faithful on the category of all finitely presented functors. Note that every faithfully flat module is in particular fp-faithfully flat.
This notion is relevant for the following situation (see \cite[Corollary 6.43]{PosBehI}):
Let $R$ be a right coherent ring and let $M \in R\Modl$ be an fp-faithfully flat module.
Then we get a commutative (up to natural isomorphism) diagram of exact functors
\begin{center}
  \begin{tikzpicture}[mylabel/.style={fill=white}]
    \coordinate (r) at (4,0);
    \coordinate (d) at (0,-2);
    \node (A) {$R\modl\modl$};
    \node (B) at ($(A) + 2*(r)$) {$\modr R$};
    \node (C) at ($(A) + (d) + (r)$) {$\ab{M}$};
    \draw[->] (A) to node[above]{$\codefect$} (B);
    \draw[->] (A) to (C);
    \draw[->] (C) to node[right, yshift=-0.5em]{$\simeq$}(B);
  \end{tikzpicture} 
\end{center}
where the unlabeled arrow is the canonical quotient functor and where the functor induced by the universal property of the Serre quotient is an equivalence.
\end{para}

\begin{lemma}\label{lemma:compute_sigma_ast}
Let $\sigma: R \rightarrow S$ be a ring morphism.
Then we have a canonical functor
\[
\sigma_{\ast}: R\modl\modl \rightarrow S\modl\modl
\]
induced by the universal property of $R\modl\modl$ applied to the functor
\[
R \xrightarrow{\sigma} S \rightarrow S\modl\modl.
\]
Concretely, if
\[
   \begin{tikzpicture}[mylabel/.style={fill=white},baseline=(A)]
        \coordinate (r) at (4.5,0);
        \node (A) {$\Hom( N, - )$};
        \node (B) at ($(A)+(r)$) {$\Hom( M, - )$};
        \node (C) at ($(B) + 0.5*(r)$) {$\mathcal{G}$};
        \node (D) at ($(C) + 0.25*(r)$) {$0$};
        \draw[->,thick] (A) to node[above]{$\Hom( \alpha, - )$} (B);
        \draw[->,thick] (B) to (C);
        \draw[->,thick] (C) to (D);
  \end{tikzpicture}
\]
is a presentation of $\mathcal{G} \in R\modl\modl$, then
\[
   \begin{tikzpicture}[mylabel/.style={fill=white},baseline=(A)]
        \coordinate (r) at (8,0);
        \node (A) {$\Hom( S \otimes_R N, - )$};
        \node (B) at ($(A)+(r)$) {$\Hom( S \otimes_R M, - )$};
        \node (C) at ($(B) + 0.5*(r)$) {$\sigma_{\ast}\mathcal{G}$};
        \node (D) at ($(C) + 0.25*(r)$) {$0$};
        \draw[->,thick] (A) to node[above]{$\Hom( S \otimes_R \alpha, - )$} (B);
        \draw[->,thick] (B) to (C);
        \draw[->,thick] (C) to (D);
  \end{tikzpicture}
\]
is a presentation of $\sigma_{\ast}(\mathcal{G})$.
\end{lemma}
\begin{proof}
This is a direct computation.
\end{proof}

\begin{para}\label{para:description_of_sigma_ast}
Let $\sigma: R \rightarrow S$ be a ring morphism and let $M \in S\Modl$.
Then the universal property of the Serre quotient induces a functor (see \cite[Section 6.7]{PosBehI})
\[
\sigma_{\ast}: \ab{M|_{\sigma}} \rightarrow \ab{M}
\]
which makes the following square of exact functors commutative up to natural isomorphism:
\begin{center}
  \begin{tikzpicture}[mylabel/.style={fill=white}]
    \coordinate (r) at (3,0);
    \coordinate (d) at (0,-2);
    
    \node (11) {$R\modl\modl$};
    \node (12) at ($(11)+(r)$) {$\ab{M|_{\sigma}}$};

    \node (21) at ($(11) + (d)$){$S\modl\modl$};
    \node (22) at ($(21)+(r)$) {$\ab{M}$};
    
    \draw[->,thick] (11) to (12);

    \draw[->,thick] (21) to (22);

    \draw[->,thick] (11) to node[left]{$\sigma_{\ast}$} (21);
    \draw[->,thick] (12) to node[right]{$\sigma_{\ast}$} (22);
  \end{tikzpicture}
\end{center}
If $\sigma$ is a ring epimorphism (this includes the case of localizations of rings), then $\sigma_{\ast}: \ab{M|_{\sigma}} \rightarrow \ab{M}$ is an equivalence of categories \cite[Theorem 6.51]{PosBehI}.
\end{para}

\begin{example}
Let $R$ be a ring and let $I \subseteq R$ be an ideal.
Let $\sigma: R \rightarrow R/I$ denote the canonical ring epimorphism.
Let $M := R/I \in R/I\Modl$.
Then we can compute
\[
\ab{ M|_{\sigma} } \simeq \ab{ M } \simeq \modr R/I
\]
since $M$ is faithfully flat in $R/I\Modl$.
\end{example}

\begin{example}
Let $R$ be a commutative ring and let $\pid \subseteq R$ be a prime ideal.
Let $\sigma: R \rightarrow R_{\pid}$ denote the canonical localization map, which is a ring epimorphism.
Let $M := R_{\pid} \in R_{\pid}\Modl$.
Then we can compute
\[
\ab{ M|_{\sigma} } \simeq \ab{ M } \simeq \modr R_{\pid}
\]
since $M$ is faithfully flat in $R_{\pid}\Modl$.
\end{example}

\section{The Ziegler spectrum}\label{section:ziegler_spectrum}

The Ziegler spectrum $\Zg(R)$ of a ring $R$ is a topological space whose closed sets classify Serre subcategories of $R\modl\modl$.
We introduce this idea using the language of (co)frames.
For a thorough treatment of the Ziegler spectrum, we refer to \cite{PrestPSL}.
Moreover, we refer to \cite[Chapter IX]{MLSMI94} or \cite{Blechschmidt} for more details on the theory of frames.

Recall that a \textbf{partially ordered set} is a set equipped with a binary relation which is reflexive, transitive, and antisymmetric.

\begin{definition}
    A \textbf{frame} is a partially ordered set $F$ with arbitrary joins ($\bigvee$) and finite meets ($\wedge$), such that the infinite distributive law
    \[
    U \wedge \bigvee_i V_i = \bigvee_i (U \wedge V_i)
    \]
    holds for all $U, V_i \in F$, where $i$ ranges over some index set.
    A frame homomorphism is a monotone map which preserves arbitrary joins and finite meets.
    We denote the resulting category by $\Frm$.
\end{definition}

\begin{para}
Let $X$ be a topological space. We set $\mathcal{O}(X) := \{ U \subseteq X \mid \text{$U$ is open} \}$.
Then $\mathcal{O}(X)$ is a frame with partial order given by inclusion of subsets.
A continuous map $f: X \rightarrow Y$ defines a frame homomorphism $V \mapsto f^{-1}(V)$ of type $\mathcal{O}(Y) \rightarrow \mathcal{O}(X)$.
The assignment $X \mapsto \mathcal{O}(X)$ defines a functor $\Top \rightarrow \Frm^{\op}$, where $\Top$ is the category of topological spaces and continuous maps.
A frame is called \textbf{spatial} if it lies in the image of this functor.
\end{para}

\begin{para}\label{para:from_to_top}
We also get a functor in the other direction: $\Frm^{\op} \rightarrow \Top$.
For this, let $\{ 0 \leq 1 \}$ be the initial object in $\Frm$.
For any frame $F$, the set $\Hom_{\Frm}( F, \{ 0 \leq 1 \} )$ can be regarded as a topological space whose opens are given by $\{ \alpha \mid \alpha(U) = 1 \}$ for $U \in F$. This construction yields the desired functor.
The topological spaces in the image of this functor are precisely the \textbf{sober} spaces.
Moreover, we have an adjunction
\[
\mathcal{O}(-) \dashv \Hom_{\Frm}(-,\{ 0 \leq 1 \}),
\]
i.e., $\mathcal{O}(-)$ is left adjoint to $\Hom_{\Frm}(-,\{ 0 \leq 1 \})$.
This adjunction is idempotent, i.e., its associated monad is idempotent. It follows that this idempotent adjunction restricts to an equivalence of categories between sober spaces and spatial frames.
The associated monad of this adjunction (the so-called \textbf{sobrification}) applied to a topological space $X$ is in bijection with the set 
\[\Irr(X) := \{ Y \mid \text{$\emptyset \neq Y \subseteq X$ is closed and irreducible} \}\] as follows:
\begin{align*}
   \Irr(X) &\xrightarrow{ \cong } \Hom_{\Frm}(\mathcal{O}(X),\{ 0 \leq 1 \}) \\
   Y &\mapsto \big( U \mapsto  \begin{cases}
    1 & \text{$Y \cap U \neq \emptyset$} \\
    0 & \text{otherwise}
  \end{cases} \big)
\end{align*}
Via this bijection, the unit of this adjunction is given by
\begin{align*}
   X  &\rightarrow \Irr(X) \\
   P &\mapsto \overline{P},
\end{align*}
where $\overline{P}$ is the closure of a point $P \in X$.
This unit is injective if and only if $X$ is a 
\textbf{Kolmogorov space}, i.e., it is $T_0$, i.e., every pair of distinct points is topologically distinguishable, i.e., if $P \in X$ and $Q \in X$ have the same open neighborhoods, then $P = Q$.
We say that $X$ \textbf{has enough points} if this unit is surjective.
Since the unit is an open map, $X$ is sober if this unit is a bijection and therefore a homeomorphism, i.e., if every non-empty closed irreducible subset of $X$ is the closure of a uniquely determined point $P \in X$.
\end{para}

\begin{lemma}\label{lemma:serre_subcategories_a_frame}
Let $\AC$ be a small abelian category.
The set 
\[\Serre( \AC ) := \{ \CC \subseteq \AC \mid \text{$\CC$ is a Serre subcategory} \}\]
is a frame with partial order given by inclusion.
\end{lemma}
\begin{proof}
Meets in $\Serre( \AC )$ are given by intersection, joins are given by generation, see Paragraph \Cref{para:Serre_subcategories}.
We need to show the infinite distributive law. For this, let $\{\CC_i\}_{i \in I}$ be a family of Serre subcategories of $\AC$ for some index set $I$ and let $\BC$ be a Serre subcategory of $\AC$.
It suffices to show $\BC \cap \langle \CC_i \mid i \in I \rangle \subseteq \langle \BC \cap \CC_i \mid i \in I \rangle$, since the other inclusion holds in any partially ordered set.
So, let $A \in \BC \cap \langle \CC_i \mid i \in I \rangle$. Then there exists a finite subset $J \subseteq I$ and objects $A_j \in \CC_j$ for $j \in J$ such that $A$ is an extension of those $A_j$.
Since $A$ lies in $\BC$, all its subfactors also lie in $\BC$, thus, $A_j \in \BC \cap \CC_j$ for $j \in J$ and hence $A \in \langle \BC \cap \CC_i \mid i \in I \rangle$.
\end{proof}

\begin{corollary}\label{corollary:AGJ_on_frames}
The Auslander-Gruson-Jensen duality yields an isomorphism of frames
\begin{align*}
    \Serre( R\modl\modl ) &\xrightarrow{\cong} \Serre( R^{\op}\modl\modl ) \\
    \CC &\mapsto \DAGJ(\CC) := \{ \mathcal{G} \mid \text{$\mathcal{G} \cong \DAGJ(\mathcal{H})$ for some $\mathcal{H} \in \CC$} \}
\end{align*}
\end{corollary}

\begin{definition}
    A \textbf{coframe} is a partially ordered set $C$ with arbitrary meets ($\bigwedge$) and finite joins ($\vee$), such that we have for all $X, Y_i \in C$, where $i$ ranges over some index set:
    \[
    X \vee \bigwedge_i Y_i = \bigwedge_i (X \vee Y_i).
    \]
\end{definition}

\begin{example}
If $F$ is a frame, then its opposite partially ordered set $F^{\op}$ is a coframe and vice versa.
\end{example}

\begin{example}
Let $X$ be a topological space and set $\mathcal{C}(X) := \{ Y \subseteq X \mid \text{$Y$ is closed} \}$.
Then $\mathcal{C}(X)$ is a coframe with partial order given by inclusion of subsets.
Taking the complement of a set yields an isomorphism of frames $\mathcal{C}(X)^{\op} \cong \mathcal{O}(X)$.
\end{example}

\begin{remark}
We note that joins in $\Serre( \AC )$ do not necessarily distribute over arbitrary meets. Thus, $\Serre( \AC )$ is not a coframe in general.
\end{remark}

\begin{para}\label{para:ziegler_spectrum}
Let $R$ be a ring.
In this paragraph, we describe a topological space $\Zg(R)$ called the \textbf{Ziegler spectrum} of $R$. A decisive property of $\Zg(R)$ is that we have an isomorphism of frames
\begin{equation}\label{equation:serre_ziegler}
\Serre( R\modl\modl ) \cong \mathcal{C}( \Zg(R) )^{\op}.    
\end{equation}
The Ziegler spectrum can be described as follows:
the points of $\Zg(R)$ are the isomorphism classes of indecomposable pure-injective modules $P \in R\Modl$, where a module is called \textbf{pure-injective} if the functor $(- \otimes_R P)$ is an injective object in the functor category of \emph{all} functors of type $\modr R \rightarrow \Ab$.
The closed subsets of $\Zg(R)$ are the vanishing loci defined by finitely presented functors, i.e., they are of the form
\[
\V( U ) := \{ P \in \Zg(R) \mid \text{$\mathcal{G}(P) \cong 0$ for all $\mathcal{G} \in U$} \}
\]
for $U \subseteq R\modl\modl$ a set of finitely presented functors. We call $\V(U)$ the \textbf{vanishing locus} of $U$.
Conversely, given a subset $X \subseteq \Zg(R)$, we obtain a Serre subcategory
\[
\mathbf{I}(X) := \{ \mathcal{G} \in R\modl\modl \mid \text{$\mathcal{G}(P) \cong 0$ for all $P \in X$} \}.
\]
Then the isomorphism of \Cref{equation:serre_ziegler} is given by the operators $\V$ and $\mathbf{I}$.
\end{para}

\begin{para}\label{para:duality_of_Irr}
The space $\Zg(R^{\op})$ can be regarded as the Ziegler spectrum of $R$ described in terms of right modules instead of left modules.
By Equation \eqref{equation:serre_ziegler} and \Cref{corollary:AGJ_on_frames}, we have
\begin{align*}
\mathcal{O}( \Zg(R) ) \cong \mathcal{C}( \Zg(R) )^{\op} &\cong \Serre( R\modl\modl )  \\
&\cong \Serre( R^{\op}\modl\modl ) \cong \mathcal{C}( \Zg(R^{\op}) )^{\op} \cong \mathcal{O}( \Zg(R^{\op}) ).
\end{align*}
It follows from Paragraph \Cref{para:from_to_top} that the sobrifications of $\Zg(R)$ and $\Zg(R^{\op})$ are homeomorphic.
Using the description of the sobrification via irreducible closed subsets from Paragraph \Cref{para:from_to_top}, the homeomorphism between the sobrifications
\[
D: \Irr( \Zg(R) ) \rightarrow \Irr( \Zg( R^{\op} ) )
\]
can be desribed as follows: let $X \in \Irr( \Zg(R) )$ be given as $X = \V( U )$ for some $U \subseteq R\modl\modl$, then $DX = \V( \DAGJ(U) ) \in \Irr( \Zg( R^{\op} ) )$.
\end{para}

\begin{remark}
It is an open question whether we have $\Zg(R) \cong \Zg(R^{\op})$ in general.
The Ziegler spectrum is not a Kolmogorov space in general (see \cite[8.2.93]{PrestPSL}).
It is even an open problem whether the Kolmogorov quotients (the identification of topologically indistinguishable points) of $\Zg(R)$ and $\Zg(R^{\op})$ are homeomorphic.
From the point of view of classifying Serre subcategories of $R\modl\modl$, the sobrification of $\Zg(R)$ (and hence of $\Zg(R^{\op})$) is the canonical topological space to look at.
\end{remark}

\begin{example}\label{example:topology_localpid}
Let $R$ be a DVR with maximal ideal $\mathfrak{p}$ (for the case of a Dedekind domain see \Cref{example:zg_dedekind}).
Then the indecomposable pure-injective $R$-modules are given as follows:
  \begin{enumerate}
    \item $F_n := R/\pid^n$, for $n \geq 1$,
    \item $A := \varprojlim_n{R/\pid^n}$, the limit of the sequence $(\dots \rightarrow R/\pid^{n+1} \rightarrow R/\pid^{n} \rightarrow \dots)$,
    \item $P := E(R/\pid)$, the injective hull of $R/\pid$,
    \item $Q := \Quot(R)$, the quotient field of $R$.
  \end{enumerate}
We describe all closed subsets of $\Zg(R)$.
Let $X_0 \subseteq \{ F_n \mid n \geq 1 \}$ and $X_1 \subseteq \{A,P,Q\}$.
Then $X := X_0 \uplus X_1$ is closed if and only if one of the following holds:
\begin{enumerate}
    \item $X_0$ is a finite set and $X_1$ is equal to one of the following sets:
    \begin{itemize}
        \item $\emptyset$
        \item $\{Q\}$
        \item $\{A,Q\}$
        \item $\{P,Q\}$
        \item $\{A,P,Q\}$
    \end{itemize}
    \item $X_0$ is an infinite set and $X_1 = \{A,P,Q\}$.
\end{enumerate}
\end{example}

\begin{example}\label{example:zg_dedekind}
  Let $R$ be a Dedekind domain.
  Then the indecomposable pure-injective $R$-modules are given as follows (see \cite[description below Remark 5.2.5]{PrestPSL}):
  \begin{enumerate}
    \item $F_{\pid,n} := R/\pid^n$ for $\pid \in \mSpec( R )$, $n \geq 1$,
    \item $A_{\pid} := \varprojlim_n{R/\pid^n}$ for $\pid \in \mSpec( R )$,
    \item $P_{\pid} := E( R/\pid )$, for $\pid \in \mSpec( R )$,
    \item $Q := \Quot(R)$.
  \end{enumerate}
We describe all closed subsets of $\Zg(R)$. For this, we note that $R_{\pid}$ is a DVR, and we can use \Cref{example:topology_localpid} to make the following identification:
\[\Zg( R_{\pid} ) \cong \{ F_{\pid,n}, A_{\pid},P_{\pid},Q \mid n \geq 1 \}.\]
Now, let $X \subseteq \Zg(R)$.
We set 
\begin{align*}
X^{\pid} &:= X \cap \Zg( R_{\pid} ) \text{, for $\pid \in \mSpec( R ),$}\\
X_0 &:= X \cap \{ F_{\pid,n} \mid \pid \in \mSpec( R ), n \geq 1\}.
\end{align*}
Then $X$ is closed if and only if all of the following statements hold:
\begin{enumerate}
    \item $X^{\pid}$ is closed in $\Zg( R_{\pid} )$ for all $\pid \in \mSpec( R )$, and
    \item if $X_0$ is infinite, then $Q \in X$.
\end{enumerate}
\end{example}

\begin{para}[Elementary duality]\label{para:reflexive}
Recall the map $D$ between the irreducible components of $\Zg(R)$ and $\Zg(R^{\op})$ of Paragraph \Cref{para:duality_of_Irr} and the unit from the adjunction described in Paragraph \Cref{para:from_to_top}:
\begin{center}
\begin{tikzpicture}[mylabel/.style={fill=white}]
      \coordinate (r) at (4,0);
      \node (A) {$\Irr(\Zg(R))$};
      \node (B) at ($(A) + (r)$) {$\Irr(\Zg(R^{\op}))$};
      \node (C) at ($(A) - (r)$) {$\Zg(R)$};
      \node (D) at ($(B) + (r)$) {$\Zg(R^{\op})$};
      \draw[->, out = 30, in = 180-30] (A) to node[mylabel]{$D$} (B);
      \draw[<-, out = -30, in = 180+30] (A) to node[mylabel]{$D$} (B);
      \node (t) at ($(A) + 0.5*(r)$) {$\cong$};
      \draw[->] (C) to node[above]{${\unit_R}$} (A);
      \draw[->] (D) to node[above]{${\unit_{R^{\op}}}$} (B);
\end{tikzpicture} 
\end{center}
In general, sending a point $P$ from the left to right of this diagram via
\[
\Zg(R) \ni P \mapsto \unit_{R^{\op}}^{-1}( \{ D( \unit_R( P ) \} ) ) \subseteq \Zg(R^{\op})
\]
does not define a function from $\Zg(R)$ to $\Zg(R^{\op})$, but only a relation.
A point $P \in \Zg(R)$ is called \textbf{reflexive} if its assigned subset $\unit_{R^{\op}}^{-1}( \{ D( \unit_R( P ) \} ) )$ is a singleton whose unique element we denote by $DP \in \Zg(R^{\op})$, and if the same holds for $DP $, i.e., $\unit_{R}^{-1}( \{ D( \unit_{R^{\op}}( DP ) \} ) )$ is a singleton (which is necessarily equal to $P$).
Hence, we end up with a bijection between reflexive points (called \textbf{elementary duality}):
\begin{center}
\begin{tikzpicture}[mylabel/.style={fill=white}]
      \coordinate (r) at (7,0);
      \node (A) {$\{ P \in \Zg(R) \mid \text{$P$ is reflexive}\}$};
      \node (B) at ($(A) + (r)$) {$\{ P \in \Zg(R^{\op}) \mid \text{$P$ is reflexive}\}$};
      \draw[->, out = 10, in = 180-10] (A) to node[mylabel]{$D$} (B);
      \draw[<-, out = -10, in = 180+10] (A) to node[mylabel]{$D$} (B);
      \node (t) at ($(A) + 0.5*(r)$) {$\cong$};
\end{tikzpicture} 
\end{center}
Moreover, for any $\mathcal{G} \in R\modl\modl$, we have
\begin{equation}\label{equation:DG}
     \mathcal{G}(P) \cong 0  \hspace{1em} \Longleftrightarrow \hspace{1em} \DAGJ(\mathcal{G})( DP ) \cong 0
\end{equation}
since
\[
\mathcal{G} \in \mathbf{I}(P) \hspace{1em} \Longleftrightarrow \hspace{1em}\DAGJ(\mathcal{G}) \in \DAGJ(\mathbf{I}(P)) \hspace{1em} \Longleftrightarrow \hspace{1em} \DAGJ(\mathcal{G}) \in \mathbf{I}(DP),
\]
where the last equivalence follows from the description of $D$ in Paragraph \Cref{para:duality_of_Irr}.
\end{para}

\begin{theorem}\label{theorem:flat_inj_duality}
Let $R$ be a left noetherian ring.
If $P \in \Zg(R)$ is an injective module, then it is reflexive (in the sense of Paragraph \Cref{para:reflexive}), and $DP \in \Zg(R^{\op})$ is a flat module.
Moreover, elementary duality gives a bijection
\begin{center}
\begin{tikzpicture}[mylabel/.style={fill=white}]
      \coordinate (r) at (7,0);
      \node (A) {$\{ P \in \Zg(R) \mid \text{$P$ is injective}\}$};
      \node (B) at ($(A) + (r)$) {$\{ P \in \Zg(R^{\op}) \mid \text{$P$ is flat}\}$};
      \draw[->, out = 10, in = 180-10] (A) to node[mylabel]{$D$} (B);
      \draw[<-, out = -10, in = 180+10] (A) to node[mylabel]{$D$} (B);
      \node (t) at ($(A) + 0.5*(r)$) {$\cong$};
\end{tikzpicture} 
\end{center}
\end{theorem}
\begin{proof}
This follows from \cite[Corollary 9.6]{HerzogDuality} and the comment below that corollary.
\end{proof}

\section{The vanishing locus of the kernel of the co- and contravariant defect}\label{section:vanishing_locus_kernel_defect}

An explicit knowledge of all the points in the Ziegler spectrum is not always given (as it is in the case of Dedekind domains).
However, there are closed subsets about which we can say more.
In this subsection, we describe the points in the closed subsets
$\V( \kernel( \defect ) )$ and $\V( \kernel( \codefect ) )$ and show how to compute
\[\V( \mathcal{G} ) \cap \V( \kernel( \defect ) )\]
and
\[\V( \mathcal{G} ) \cap \V( \kernel( \codefect ) )\]
in the case of a commutative noetherian ring for $\mathcal{G} \in R\modl\modl$, see \Cref{theorem:checking_kernels_of_defects}.

\begin{lemma}\label{lemma:kerdefect_injective}
Let $R$ be a left coherent ring.
Then
\[
\V( \kernel( \defect ) ) = \{ P \in \Zg(R) \mid \text{$P$ is injective} \}.
\]
\end{lemma}
\begin{proof}
Let $P \in \Zg(R)$.
Then 
\begin{align*}
\phantom{\Leftrightarrow} & \quad P \in \V( \kernel( \defect ) ) \\
\Leftrightarrow & \quad \mathbf{I}(P) \supseteq \kernel( \defect ) \\\Leftrightarrow & \quad \kernel(\ev{P}) \supseteq \kernel( \defect ) \\
\Leftrightarrow & \quad \text{$P$ is fp-injective}
\end{align*}
where the last equivalence is the statement in \cite[Theorem 6.28]{PosBehI}.
Last, a module is injective if and only if it is both fp-injective and pure-injective (see \cite[Lemma 4.3.12.]{PrestPSL}). This gives the claim.
\end{proof}

\begin{lemma}\label{lemma:Matlis_inj}
Let $R$ be a commutative noetherian ring. Then
\[
\{ P \in \Zg(R) \mid \text{$P$ is injective} \} = \{ E(R/\pid) \mid \pid \in \Spec(R) \}.
\]
\end{lemma}
\begin{proof}
Every injective module is pure-injective. Thus, the set on the left hand side of the equation is the set of isomorphism classes of all indecomposable injective modules.
Now, the claim follows from the classification of such injectives due to Matlis (see \cite[Proposition 3.1]{Matlis}).
\end{proof}

\begin{lemma}\label{lemma:kercodefect_flat}
Let $R$ be a right coherent ring.
Then
\[
\V( \kernel( \codefect ) ) = \{ P \in \Zg(R) \mid \text{$P$ is flat} \}.
\]
\end{lemma}
\begin{proof}
Let $P \in \Zg(R)$.
Then 
\begin{align*}
\phantom{\Leftrightarrow} & \quad P \in \V( \kernel( \codefect ) ) \\
\Leftrightarrow & \quad \mathbf{I}(P) \supseteq \kernel( \codefect ) \\\Leftrightarrow & \quad \kernel(\ev{P}) \supseteq \kernel( \codefect ) \\
\Leftrightarrow & \quad \text{$P$ is flat}
\end{align*}
where the last equivalence is the statement in \cite[Theorem 6.42]{PosBehI}.
\end{proof}

\begin{lemma}\label{lemma:Matlis_flat}
Let $R$ be a commutative noetherian ring. Then
\[
\{ P \in \Zg(R) \mid \text{$P$ is flat} \} = \{ D(E(R/\pid)) \mid \pid \in \Spec(R) \}
\]
\end{lemma}
\begin{proof}
Follows from \Cref{lemma:Matlis_inj} and \Cref{theorem:flat_inj_duality}.
\end{proof}

\begin{lemma}\label{lemma:defect_prime}
Let $R$ be a coherent commutative ring, let $\pid \subseteq R$ be a prime ideal, and let $\sigma: R \rightarrow R_{\pid}$ denote the canonical localization map.
If $M \in R_{\pid}\Modl$ is an fp-injective fp-cogenerator (in the category $R_{\pid}\Modl$), then the following diagram of exact functors commutes up to natural isomorphism:
\begin{center}
  \begin{tikzpicture}[mylabel/.style={fill=white}]
    \coordinate (r) at (6,0);
    \coordinate (d) at (0,-2);
    
    \node (11) {$R\modl\modl$};
    
    \node (13) at ($(11) + 2*(r)$) {$(R\modl)^{\op}$};

    \node (21) at ($(11) + (d)$){$\frac{R\modl\modl}{\kernel( \ev{\M_{|\ringmap}} ) } = \ab{\M_{|\ringmap}}$};
    \node (22) at ($(21)+(r)$) {$\ab{M} = \frac{R_{\pid}\modl\modl}{\kernel( \ev{\M} ) }$};
    \node (23) at ($(22) + (r)$) {$(R_{\pid}\modl)^{\op}$};

    \draw[->,thick] (11) to node[above]{$\mathcal{G} \mapsto \defect( \mathcal{G})$} node[below]{$\simeq$} (13);

    \draw[->,thick] (21) to node[below]{$\ringmap_{\ast}$} node[above]{$\simeq$} (22);
    \draw[->,thick] (22) to node[below]{$\mathcal{G} \mapsto \defect( \mathcal{G})$} node[above]{$\simeq$}(23);

    \draw[->,thick] (11) to (21);
    
    \draw[->,thick] (13) to node[left]{$(R_{\pid} \otimes_R -)$}(23);
  \end{tikzpicture}
\end{center}
\end{lemma}
\begin{proof}
We note that $R_{\pid}$ is a coherent ring and flat as an $R$-module. Thus, all functors are well-defined and exact, and $\sigma_{\ast}$ is an equivalence by Paragraph \Cref{para:description_of_sigma_ast}.
By the universal property of $R\modl\modl$ (see Paragraph \Cref{para:up_of_Rmodmod}), it suffices to test commutativity on the full subcategory spanned by the forgetful functor $\ff$.
Mapping $\ff$ via the upper path of the diagram yields $R_{\pid} \otimes_R \defect( \ff ) \cong R_{\pid} \otimes R \cong R_{\pid}$.
Mapping $\ff$ via the lower path of the diagram yields $\defect( \ff' ) \cong R_{\pid}$, where $\ff'$ is the forgetful functor $R_{\pid}\modl \rightarrow \Ab$. The claim follows.
\end{proof}

\begin{lemma}\label{lemma:defect_codefect_duality}
For $\mathcal{G} \in R\modl\modl$, we have
\[
\defect( \DAGJ(\mathcal{G} )) \simeq \codefect( \mathcal{G} ) \in \modr R.
\]
\end{lemma}
\begin{proof}
This is clear for $\mathcal{G} = (N \otimes_R -)$, where $N \in \modr R$.
The general case follows from the exactness of the involved functors.
\end{proof}

\begin{theorem}\label{theorem:checking_kernels_of_defects}
Let $R$ be a noetherian commutative ring and let $\pid \subseteq R$ be a prime ideal.
We have for all $\mathcal{G} \in R\modl\modl$:
    \[
    {\mathcal{G}}(E( R/\pid) ) \cong 0 \text{\hspace{2em}$\Longleftrightarrow$ \hspace{2em}} R_{\pid} \otimes_R \defect( \mathcal{G} ) \cong 0.
    \]
In particular, \Cref{lemma:kerdefect_injective} and \Cref{lemma:Matlis_inj} imply the equality
\[
\V( \mathcal{G} ) \cap \V( \kernel( \defect ) ) = \{ E( R/\qid ) \mid \text{$R_{\qid} \otimes_R \defect( \mathcal{G} ) \cong 0$, $\qid \in \Spec(R)$} \}.
\]
Moreover, we have
    \[
    {\mathcal{G}}(D(E( R/\pid)) ) \cong 0 \text{\hspace{2em}$\Longleftrightarrow$ \hspace{2em}}  \codefect( \mathcal{G} ) \otimes_R R_{\pid} \cong 0.
    \]
In particular, \Cref{lemma:kercodefect_flat} and \Cref{lemma:Matlis_flat} imply the equality
\[
\V( \mathcal{G} ) \cap \V( \kernel( \codefect ) ) = \{ D(E( R/\qid )) \mid \text{$\codefect( \mathcal{G} ) \otimes_R R_{\qid} \cong 0$, $\qid \in \Spec(R)$}  \}.
\]
\end{theorem}
\begin{proof}
Let $\sigma: R \rightarrow R_{\pid}$ be the canonical localization map.
We set $M := E( R_{\pid}/\pid_{\pid} ) \in R_{\pid}\Modl$ and note that $E(R/\pid) \cong M|_{\sigma}$ (see, e.g., \cite[Theorem 18.14 (vi)]{Matsumura}).
We note that $M$ is an injective $R_{\pid}$-module by construction.
Moreover, it is a cogenerator of $R_{\pid}\Modl$ (see, e.g., \cite[Theorem 19.10]{Lam}).
Thus, we can apply \Cref{lemma:defect_prime}:
we have ${\mathcal{G}}(E( R/\pid) ) \cong 0$ if and only if $\mathcal{G} \in \kernel( \ev{\M_{|\ringmap}} )$ if and only if
$\defect( \sigma_{\ast}(\asBeh{\mathcal{G}}{\M_{|\ringmap}} ) ) \cong 0$ (following the lower path in the diagram of \Cref{lemma:defect_prime}) if and only if 
$R_{\pid} \otimes_R \defect( \mathcal{G} ) \cong 0$ (following the upper path in the diagram of \Cref{lemma:defect_prime}).
For the statement about the covariant defect, we compute
\begin{align*}
{\mathcal{G}}(D(E( R/\pid)) ) \cong 0 \hspace{1em}&\Longleftrightarrow \hspace{1em}{(\DAGJ\mathcal{G})}(E( R/\pid) ) \cong 0 \\
&\Longleftrightarrow \hspace{1em} \defect( \DAGJ(\mathcal{G}) ) \otimes_R R_{\pid}   \cong 0 \\
&\Longleftrightarrow \hspace{1em} \codefect( \mathcal{G} ) \otimes_R R_{\pid} \cong 0,
\end{align*}
where we use Equation \eqref{equation:DG}, the corresponding statement about the contravariant defect, and \Cref{lemma:defect_codefect_duality}.
\end{proof}

\begin{remark}\label{remark:checking_kernels_of_defects}
Let $R$ be a noetherian commutative ring and let $M \in R\modl$.
Then its support can be computed via its annihilator, i.e., we have
\[
\Supp(M) := \{ \pid \in \Spec(R) \mid R_{\pid} \otimes M \ncong  0 \} = \{ \pid \in \Spec(R) \mid \pid \supseteq \Ann( M  ) \}
\]
where
\[
\Ann(M) := \{ r \in R \mid rM = 0 \}.
\]
Thus, \Cref{theorem:checking_kernels_of_defects} tells us that we have canonical bijections
\[
\V( \mathcal{G} ) \cap \V( \kernel( \defect ) ) \cong \Spec(R) \setminus \Supp( \defect( \mathcal{G} ) )
\]
and
\[
\V( \mathcal{G} ) \cap \V( \kernel( \codefect ) ) \cong \Spec(R) \setminus \Supp( \codefect( \mathcal{G} ) ).
\]
We note that if $R$ is a computable ring in the sense of \cite{BarakatLHAxiomatic}, the annihilator of a finitely presented module can be computed by means of computer algebra.
\end{remark}

\section{The Grothendieck group of \texorpdfstring{$R\modl\modl$}{R-mod-mod}}\label{section:grothendieck_group_of_free_abelian_cat}

In this section, we give a description of the Grothendieck group of $R\modl\modl$ regarded as an \emph{abelian} category.
For this, we need to start with an introduction to the Grothendieck group of an \emph{additive} category.

\begin{definition}\label{def:additive_context_additive}
Let $\AC$ be an \emph{additive} category. A map $\alpha: \obj( \AC ) \rightarrow T$, where $T$ is an abelian group, is called an \textbf{additive invariant (for $\AC$)} if for all 
$A,B,C \in \AC$ with $C \cong A \oplus B$, we have
\[\alpha(C) = \alpha(A) + \alpha(B).\]
\end{definition}

\begin{definition}
Let $\AC$ be an additive category.
Its \textbf{Grothendieck group} consists of
\begin{enumerate}
    \item an abelian group $K_0( \AC )$,
    \item an additive invariant $A \mapsto [A]$ of type $\obj(\AC) \rightarrow K_0( \AC )$
\end{enumerate}
which satisfies the following universal property: for every abelian group $T$ and additive invariant $\alpha: \obj( \AC ) \rightarrow T$, there exists a uniquely determined morphism of abelian groups $K_0( \AC ) \rightarrow T$ such that the following diagram of maps commutes:
\begin{center}
       \begin{tikzpicture}[label/.style={postaction={
        decorate,
        decoration={markings, mark=at position .5 with \node #1;}},
        mylabel/.style={thick, draw=none, align=center, minimum width=0.5cm, minimum height=0.5cm,fill=white}}]
        \coordinate (r) at (4,0);
        \coordinate (d) at (0,-2);
        \node (A) {$\obj( \AC )$};
        \node (B) at ($(A)+(r)$) {$K_0( \AC )$};
        \node (C) at ($(B) + (d)$) {$T$};
        \draw[->,thick] (A) to (B);
        \draw[->,thick] (A) to (C);
        \draw[->,thick,dashed] (B) to node[right]{$\exists!$} (C);
        \end{tikzpicture}
\end{center}
\end{definition}

\begin{para}
The question whether an additive category $\AC$ has a Grothendieck group leads to set-theoretic considerations.
If $\AC$ is a small additive category, then we can construct its Grothendieck group as the group generated by the symbols $[A]$ for $A \in \AC$ subject to the relations $[C] = [A] + [B]$ for all $A,B,C \in \AC$ with $C \cong A \oplus B$.
\end{para}

Next, we give an introduction to the Grothendieck group of an abelian category.

\begin{definition}\label{def:additive_context_abelian}
Let $\AC$ be an \emph{abelian} category. A map $\alpha: \obj( \AC ) \rightarrow T$, where $T$ is an abelian group, is called an \textbf{additive invariant (for $\AC$)} if for all short exact sequences $0 \rightarrow A \rightarrow C \rightarrow B \rightarrow 0$ in $\AC$, we have
\[\alpha(C) = \alpha(A) + \alpha(B).\]
\end{definition}

\begin{remark}
If $\AC$ is an abelian category, then we denote its underlying additive category by $|\AC|$ (see \Cref{notation:cats}).
In this way, we can resolve the clash of the names in \Cref{def:additive_context_additive} and \Cref{def:additive_context_abelian} as follows:
an additive invariant for $|\AC|$ refers to \Cref{def:additive_context_additive} while an additive invariant for $\AC$ refers to \Cref{def:additive_context_abelian}.
\end{remark}

\begin{definition}
Let $\AC$ be an abelian category.
Its \textbf{Grothendieck group} consists of
\begin{enumerate}
    \item an abelian group $K_0( \AC )$,
    \item an additive invariant $A \mapsto [A]$ of type $\obj(\AC) \rightarrow K_0( \AC )$
\end{enumerate}
which satisfies the following universal property: for every abelian group $T$ and additive invariant $\alpha: \obj( \AC ) \rightarrow T$, there exists a uniquely determined morphism of abelian groups $K_0( \AC ) \rightarrow T$ such that the following diagram commutes:
\begin{center}
       \begin{tikzpicture}[label/.style={postaction={
        decorate,
        decoration={markings, mark=at position .5 with \node #1;}},
        mylabel/.style={thick, draw=none, align=center, minimum width=0.5cm, minimum height=0.5cm,fill=white}}]
        \coordinate (r) at (4,0);
        \coordinate (d) at (0,-2);
        \node (A) {$\obj( \AC )$};
        \node (B) at ($(A)+(r)$) {$K_0( \AC )$};
        \node (C) at ($(B) + (d)$) {$T$};
        \draw[->,thick] (A) to (B);
        \draw[->,thick] (A) to (C);
        \draw[->,thick,dashed] (B) to node[right]{$\exists!$} (C);
        \end{tikzpicture}
\end{center}
\end{definition}

\begin{para}
The question whether an abelian category $\AC$ has a Grothendieck group leads to set-theoretic considerations.
If $\AC$ is a small abelian category, then we can construct its Grothendieck group as the group generated by the symbols $[A]$ for $A \in \AC$ subject to the relations $[C] = [A] + [B]$ for all short exact sequences $0 \rightarrow A \rightarrow C \rightarrow B \rightarrow 0$ in $\AC$.
\end{para}

\begin{para}
An exact functor $\AC \xrightarrow{F} \BC$ between small abelian categories gives a group morphism 
\begin{align*}
K_0(\AC) &\xrightarrow{K_0(F)} K_0(\BC) \\
[A] &\mapsto [F(A)].
\end{align*}
\end{para}

\begin{example}\label{example:length}
Let $R$ be a ring and let $M \in R\Modl$.
Recall that a composition series of $M$ is given by a chain of submodules
\[
M_0 = M \supsetneq M_1 \supsetneq \dots \supsetneq M_n = 0
\]
for some $n \geq 0$ such that each factor $M_i/M_{i+1}$ is a simple module, for $i = 0, \dots, n-1$. If $M$ has a composition series, then the number $\length_R(M) := n$ is uniquely determined and called the \textbf{length} of $M$. In that case, we say that $M$ is of finite length.
The function
\[
M \mapsto \mathrm{length}_R(M) \in \Z
\]
is additive for the full abelian subcategory $(R\Modl)_{\mathrm{fl}} \subseteq R\Modl$ spanned by modules of finite length.
Moreover, $K_0( (R\Modl)_{\mathrm{fl}} )$ is freely generated by the isomorphism classes of simple modules, see, e.g., \cite[Exercise 6.1]{WeibelKBook}.
In particular, if $R$ is a left artinian ring, a consequence of the Akizuki–Hopkins–Levitzki theorem is that 
\[
(R\Modl)_{\mathrm{fl}} = R\modl
\]
and thus $K_0( R\modl )$ is freely generated by the isomorphism classes of simple modules.
\end{example}

\begin{theorem}[Resolution theorem]\label{theorem:resolution_theorem}
Let $\AC$ be an abelian category such that every object has a projective resolution of finite length. We have an isomorphism of groups
\begin{align*}
    K_0( \Proj( \AC ) ) &\xrightarrow{\cong} K_0( \AC ) \\
    [P] &\mapsto [P]
\end{align*}
The inverse of this isomorphism maps $[A] \in K_0(\AC)$ to its Euler characteristic
\[
\sum_{i = 0}^n (-1)^i[P_i] \in K_0( \Proj( \AC ) ),
\]
where
  \begin{center}
  \begin{tikzpicture}[mylabel/.style={fill=white}]
      \coordinate (r) at (2,0);
      \node (A) {$0$};
      \node (B) at ($(A)+(r)$) {$A$};
      \node (C) at ($(B) + (r)$) {$P_0$};
      \node (D) at ($(C) + (r)$) {$P_1$};
      \node (E) at ($(D) + (r)$) {$\dots$};
      \node (F) at ($(E) + (r)$) {$P_n$};
      \draw[->,thick] (F) to (E);
      \draw[->,thick] (E) to (D);
      \draw[->,thick] (D) to (C);
      \draw[->,thick] (C) to (B);
      \draw[->,thick] (B) to (A);
  \end{tikzpicture}
\end{center}
is a projective resolution of length $n \geq 0$.
\end{theorem}
\begin{proof}
This resolution theorem goes back to Grothendieck. Its proof can be found in \cite[Theorem VIII.4.2]{BassKTheory}, where it is written in the more general case for suitable exact subcategories $\CC_0 \subseteq \CC \subseteq \AC$, which we can specialize to our case of interest by setting $\CC_0 := \Proj( \AC )$ and $\CC := \AC$.

\end{proof}

\begin{definition}\label{definition:lmc_rmc}
Let $R$ be a ring and let $\mathcal{G} \in R\modl\modl$ have a presentation given by $M \xrightarrow{\alpha} N$ in $R\modl$.
We define
\[
\lmc( \mathcal{G} ) := [M] - [N] + [\cokernel(\alpha)] \in K_0( |R\modl| )
\]
and call this element the \textbf{left module class} of $\mathcal{G}$.
Dually, let $\mathcal{G}$ have a copresentation given by $M' \xrightarrow{\beta} N'$ in $\modr R$ (see \Cref{definition:copresentations}).
We define
\[
\rmc( \mathcal{G} ) := [M'] - [N'] + [\cokernel(\beta)] \in K_0( |\modr R| )
\]
and call this element the \textbf{right module class} of $\mathcal{G}$.
\end{definition}

\begin{theorem}\label{theorem:K_0_fpfunctors}
Let $R$ be a ring. Then we have isomorphisms of groups
\begin{align*}
    K_0( R\modl\modl ) &\cong K_0( |R\modl| ) \\
    [\mathcal{G}] &\mapsto \lmc( \mathcal{G} )\\
    [ \Hom( M, - ) ] &\mapsfrom [M]
\end{align*}
and
\begin{align*}
    K_0( R\modl\modl ) &\cong K_0( |\modr R| ) \\
    [\mathcal{G}] &\mapsto \rmc(\mathcal{G})\\
    [(M\otimes -)] &\mapsfrom [M]
\end{align*}
\end{theorem}
\begin{proof}
For the first isomorphism, by \Cref{lemma:projectives_injectives_in_Rmodmod}, we have $\Proj( R\modl\modl ) \simeq (R\modl)^{\op}$.
By \Cref{lemma:projectives_length2}, every object in $R\modl\modl$ has a finite projective resolution. Thus, we can apply \Cref{theorem:resolution_theorem} and obtain:
\begin{align*}
    K_0( R\modl\modl ) &\cong K_0( \Proj( R\modl\modl ) ) \\
    &\cong K_0( |(R\modl)^{\op}| ) \cong K_0( |(R\modl)| ).
\end{align*}
Mapping an element $[\mathcal{G}] \in K_0( R\modl\modl )$ via these isomorphisms yields the desired formula.
The second isomorphism is obtained by the Auslander-Gruson-Jensen duality.
\end{proof}

\begin{remark}
In particular, \Cref{theorem:K_0_fpfunctors} implies that the left and right module classes of a finitely presented functor are independent of the chosen presentation and copresentation.
\end{remark}

The left and right module classes are compatible with the defects in the following sense:

\begin{lemma}\label{lemma:lmc_defect_rcm_codefect}
Let $R$ be ring and let $\mathcal{G} \in R\modl\modl$.
If $R$ is left coherent, then the following diagram commutes:
\begin{center}
  \begin{tikzpicture}[mylabel/.style={fill=white}]
    \coordinate (r) at (4,0);
    \coordinate (d) at (0,-2);
    \node (A) {$K_0( | R\modl | )$};
    \node (B) at ($(A) + 2*(r)$) {$K_0( R\modl\modl )$};
    \node (C) at ($(A) + (d) + (r)$) {$K_0( (R\modl)^{\op} )$};
    \draw[<-] (A) to node[above]{$\lmc(\mathcal{G}) \mapsfrom [\mathcal{G}]$}(B);
    \draw[->] (A) to node[left,xshift=-1em]{$[M] \mapsto [M]$}(C);
    \draw[->] (B) to node[right,xshift=1em]{$K_0(\defect)$}(C);
  \end{tikzpicture} 
\end{center}
Dually, if $R$ is right coherent, then the following diagram commutes:
\begin{center}
  \begin{tikzpicture}[mylabel/.style={fill=white}]
    \coordinate (r) at (4,0);
    \coordinate (d) at (0,-2);
    \node (A) {$K_0( | \modr R | )$};
    \node (B) at ($(A) + 2*(r)$) {$K_0( R\modl\modl )$};
    \node (C) at ($(A) + (d) + (r)$) {$K_0( \modr R )$};
    \draw[<-] (A) to node[above]{$\rmc(\mathcal{G}) \mapsfrom [ \mathcal{G} ]$}(B);
    \draw[->] (A) to node[left,xshift=-1em]{$[N] \mapsto [N]$}(C);
    \draw[->] (B) to node[right,xshift=1em]{$K_0( \codefect )$}(C);
  \end{tikzpicture} 
\end{center}
\end{lemma}
\begin{proof}
We have $\defect( \Hom(M,-) ) \cong M$ for $M \in R\modl$ and $\codefect( (N \otimes -) ) \cong N$ for $N \in \modr R$. The claim follows.
\end{proof}

The following result explains how to convert a left module class into a right module class and vice versa.

\begin{corollary}\label{corollary:K_0_of_left_and_right_modules}
Let $R$ be a ring. 
Combining the two isomorphisms from \Cref{theorem:K_0_fpfunctors} yields an isomorphism of groups
\begin{align*}
    K_0( |R \modl| ) \cong K_0( |\modr R| )
\end{align*}
which is given as follows: for $M \in R\modl$ 
let
\begin{center}
  \begin{tikzpicture}[mylabel/.style={fill=white}]
      \coordinate (r) at (3,0);
      
      \node (A) {$Q$};
      \node (B) at ($(A)+(r)$) {$P$};
      \node (C) at ($(B) + (r)$) {$M$};
      \node (D) at ($(C) + (r)$) {$0$};
      
      \draw[->,thick] (A) to (B);
      \draw[->,thick] (B) to (C);
      \draw[->,thick] (C) to (D);
  \end{tikzpicture}
\end{center}
be an exact sequence with $P,Q$ finitely presented projective $R$-modules.
Then this isomorphism sends $[M]$ to 
\[
[\Tr(M)] - [Q^{\ast}] + [P^{\ast}]
\]
where $\Tr(M)$ is the Auslander transpose of $M$ (see Paragraph \Cref{para:Auslander_transpose}).
\end{corollary}
\begin{proof}
We use the injective resolution of the hom functor in \Cref{lemma:proj_res_of_tensor_functor}.
\end{proof}

\begin{example}\label{example:conversion_of_projectives}
The isomorphism from \Cref{corollary:K_0_of_left_and_right_modules} maps $[P]$ to $[P^{\ast}]$ for $P \in \Proj( R\modl )$.
In particular, we see that even in the case when $R$ is commutative, the isomorphism from \Cref{corollary:K_0_of_left_and_right_modules} is not necessarily given by simply regarding the class of a left module as a right module.
\end{example}

\begin{remark}
In general, the categories $R\modl$ and $\modr R$ are neither equivalent nor antiequivalent. Hence, the group isomorphism in \Cref{corollary:K_0_of_left_and_right_modules} cannot be lifted to an (anti)equivalence of categories in a naive way. However, we can give a functorial explanation of this isomorphism using the language of triangulated categories, see \Cref{lemma:triangulated}.
\end{remark}

\begin{para}\label{para:triangulated}
We shortly explain the notions relevant for the following \Cref{lemma:triangulated}.
For an additive category $\AC$, we denote by $\Kb( \AC )$ the homotopy category of the category of bounded cochain complexes of $\AC$.
It is a triangulated category.
Moreover, if $\AC$ is an abelian category, we denote by $\Db( \AC )$ its bounded derived category. It also is a triangulated category.
Recall that if every object in $\AC$ has a finite projective resolution, then we have an exact equivalence of triangulated categories $\Db( \AC ) \simeq \Kb( \Proj( \AC ) )$. Dually, if every object in $\AC$ has a finite injective resolution, then we have an exact equivalence of triangulated categories $\Db( \AC ) \simeq \Kb( \Inj( \AC ) )$.

For $\TC$ a small triangulated category we have a corresponding notion of its Grothendieck group $K_0(\TC)$: it is the group spanned by the isomorphism classes $[T]$ of objects in $\TC$ modulo the relations $[T'] + [T''] = [T]$ for each distinguished triangle $T' \rightarrow T \rightarrow T''$ in $\TC$.
If $\AC$ is an additive category, then we have an isomorphism of groups
\[
K_0( \AC ) \cong K_0( \Kb( \AC ) )
\]
whose two directions are given by taking the Euler characteristic of a complex and by regarding an object of $\AC$ as a complex concentrated in degree $0$, see, e.g., \cite{Rose}.
\end{para}

\begin{lemma}\label{lemma:triangulated}
Let $R$ be a ring.
We have an exact equivalence of triangulated categories:
\begin{align*}
\Kb( R \modl )^{\op} \simeq \Kb( \modr R ).
\end{align*}
Concretely, for $M \in R\modl$ let
\begin{center}
  \begin{tikzpicture}[mylabel/.style={fill=white}]
      \coordinate (r) at (3,0);
      
      \node (A) {$Q$};
      \node (B) at ($(A)+(r)$) {$P$};
      \node (C) at ($(B) + (r)$) {$M$};
      \node (D) at ($(C) + (r)$) {$0$};
      
      \draw[->,thick] (A) to (B);
      \draw[->,thick] (B) to (C);
      \draw[->,thick] (C) to (D);
  \end{tikzpicture}
\end{center}
be an exact sequence with $P,Q$ finitely presented projective $R$-modules.
Then this equivalence sends $M$ (regarded as a cochain complex concentrated in degree $0$) to the cochain complex
\begin{center}
  \begin{tikzpicture}[mylabel/.style={fill=white}]
      \coordinate (r) at (2.5,0);
      
      \node (A) {$\dots$};
      \node (B) at ($(A)+(r)$) {$0$};
      \node (C) at ($(B) + (r)$) {$P^{\ast}$};
      \node (D) at ($(C) + (r)$) {$Q^{\ast}$};
      \node (E) at ($(D) + (r)$) {$\Tr(M)$};
      \node (F) at ($(E) + (r)$) {$0$};
      \node (G) at ($(F) + (r)$) {$\dots$};
      
      \draw[->,thick] (A) to (B);
      \draw[->,thick] (B) to (C);
      \draw[->,thick] (C) to (D);
      \draw[->,thick] (D) to  (E);
      \draw[->,thick] (E) to (F);
      \draw[->,thick] (F) to (G);
  \end{tikzpicture}
\end{center}
concentrated in degrees $0$, $1$, and $2$, where $\Tr(M)$ is the Auslander transpose of $M$ (in degree $2$).
In particular, the corresponding isomorphism of Grothendieck groups
\begin{align*}
K_0( |R \modl| ) \cong K_0(\Kb( R \modl )^{\op} ) \cong K_0(\Kb( \modr R )) \cong K_0( |\modr R| )
\end{align*}
can be identified with the isomorphism of \Cref{corollary:K_0_of_left_and_right_modules}.
\end{lemma}
\begin{proof}
We have the following exact equivalences of triangulated categories:
\begin{align*}
\Kb( (R \modl)^{\op} ) &\simeq \Kb( \Proj(R\modl\modl ) )\\
&\simeq \Db( R\modl\modl ) \\
&\simeq \Kb( \Inj(R\modl\modl ) ) \simeq \Kb( \modr R ).
\end{align*}
The first and the last equivalence follow from \Cref{lemma:projectives_injectives_in_Rmodmod}.
The other equivalences follow from Paragraph \Cref{para:triangulated}.
Moreover, for every additive category $\AC$, we have $\Kb( \AC^{\op} ) \simeq \Kb( \AC )^{\op}$.
The concrete formula for the equivalence follows from the injective resolution of the hom functor described in \Cref{lemma:proj_res_of_tensor_functor}.
\end{proof}

\section{Additive invariants for the detection of being zero}\label{section:additive_invariants_for_the_detection}

In this section, we show that in some cases, there are additive invariants which can detect whether a finitely presented functor vanishes at a given module or not.

\begin{definition}
Let $\AC$ be an abelian category. 
We say that an additive invariant $\alpha: \AC \rightarrow T$ \textbf{detects being zero} if the following holds:
\[
\text{$\alpha( A ) = 0$ if and only if $A \cong 0$}
\]
for all $A \in \AC$.
We say $K_0( \AC )$ \textbf{detects being zero} if the following holds:
\[
\text{$[A] = 0$ if and only if $A \cong 0$}
\]
for all $A \in \AC$.
\end{definition}

\begin{lemma}
Let $\AC$ be a abelian category.
There exists an additive invariant that detects being zero if and only if $K_0(\AC)$ detects being zero.
\end{lemma}
\begin{proof}
Every additive invariant of $\AC$ factors over $K_0(\AC)$.
\end{proof}

\begin{example}\label{example:length_detects_being_zero}
Let $R$ be a left artinian ring. Then $K_0(R\modl)$ detects being zero since $\length_R$ detects being zero, see \Cref{example:length}.
\end{example}

\begin{definition}
Let $R$ be a ring.
We say that an additive invariant
\[
\alpha: \obj( R\modl\modl ) \rightarrow T
\]
\textbf{respects being zero at $M \in R\Modl$} if the following holds:
\[
\text{$\mathcal{G}(M) \cong 0$ implies $\alpha( \mathcal{G} ) = 0$}
\]
for all $\mathcal{G} \in R\modl\modl$.
We say that $\alpha$ \textbf{detects being zero at $M$} if the following holds:
\[
\text{$\alpha( \mathcal{G} ) = 0$ if and only if $\mathcal{G}(M) \cong 0$}
\]
for all $\mathcal{G} \in R\modl\modl$.
\end{definition}

\begin{example}
The additive invariant given by the canonical quotient map
\[
K_0( R\modl\modl ) \rightarrow \frac{K_0( R\modl\modl )}{\langle [\mathcal{G}] \mid \mathcal{G}(M) \cong 0 \rangle } 
\]
respects being zero at $M \in R\Modl$ by construction. Moreover, it clearly is the universal (initial) additive invariant with this property.
\end{example}

\begin{theorem}(Localization theorem)\label{theorem:localisation_thm}
Let $\AC$ be a small abelian category with $\CC \subseteq \AC$ a Serre subcategory. Then we have an exact sequence of abelian groups:
\begin{center}
  \begin{tikzpicture}[mylabel/.style={fill=white}]
      \coordinate (r) at (4,0);
      
      \node (A) {$K_0( \CC )$};
      \node (B) at ($(A)+(r)$) {$K_0( \AC )$};
      \node (C) at ($(B) + (r)$) {$K_0( \frac{\AC}{\CC} )$};
      \node (D) at ($(C) + (r)$) {$0$.};
      
      \draw[->,thick] (A) to node[above]{$[C] \mapsto [C]$} (B);
      \draw[->,thick] (B) to node[above]{$[A] \mapsto [A]$} (C);
      \draw[->,thick] (C) to (D);
  \end{tikzpicture}
\end{center}
\end{theorem}
\begin{proof}
See \cite[Theorem 6.4]{WeibelKBook}, where Weibel attributes this theorem to Heller.
\end{proof}

\begin{corollary}\label{corollary:K0_of_abM}
Let $R$ be a ring and let $M \in R\Modl$. The canonical quotient map
\[
K_0( R\modl\modl ) \xrightarrow{[\mathcal{G}] \mapsto [ \asBeh{\mathcal{G}}{\M} ]} K_0( \ab{M} )
\]
induces an isomorphism
\[
K_0( \ab{M} ) \cong \frac{K_0( R\modl\modl )}{\langle [\mathcal{G}] \mid \mathcal{G}(M) \cong 0 \rangle }.
\]
\end{corollary}
\begin{proof}
We apply \Cref{theorem:localisation_thm}.
\end{proof}

\begin{corollary}\label{corollary:abM_detecs_zero_iff_there_exists}
Let $R$ be a ring and let $M \in R\Modl$.
There exists an additive invariant which detects being zero at $M$ if and only if $K_0( \ab{M} )$ detects being zero.
\end{corollary}
\begin{proof}
An additive invariant which respects being zero factors over $K_0( \ab{M} )$.
\end{proof}

\begin{definition}
Let $R$ be a ring and let $\sigma: R \rightarrow S$ be a ring morphism into a right artinian ring $S$.
We note that, since $S$ is right artinian, it is right coherent and thus $\codefect: S\modl\modl \rightarrow \modr S$ is well-defined. Moreover, since $S$ is right artinian, the length gives an additive invariant $K_0(\modr S) \xrightarrow{\length_S} \Z$, see \Cref{example:length}.
Thus, we can define for $\mathcal{G} \in R\modl\modl$:
\[
\length_{\sigma}( \mathcal{G} ) := \length_{S}( \codefect( \sigma_{\ast}( \mathcal{G} ) ) ) \in \Z.
\]
In other words, $\length_{\sigma}$ is the additive invariant obtained by the following composition:
\[
K_0(R\modl\modl) \xrightarrow{K_0(\sigma_{\ast})} K_0(S\modl\modl) \xrightarrow{K_0(\codefect)} K_0(\modr S) \xrightarrow{\length_S} \Z.
\]

\end{definition}

\begin{theorem}\label{theorem:K_0_zero_detection}
Let $R$ be a ring and let $\sigma: R \rightarrow S$ be a ring morphism into a right artinian ring $S$.
If $M \in S\Modl$ is an fp-faithfully flat $S$-module, then $\length_{\sigma}$ detects being zero at $M_{|\sigma}$. In particular, $\length_{\sigma}$ detects being zero at $S_{|\sigma} \in R\Modl$, i.e., at $S$ regarded as a left $R$-module via the map $\sigma$.
\end{theorem}
\begin{proof}
We note that by Paragraph \Cref{para:codefect_equiv} and Paragraph \Cref{para:description_of_sigma_ast}, we have a commutative diagram (up to natural isomorphism) of exact functors
\begin{center}
  \begin{tikzpicture}[mylabel/.style={fill=white}]
    \coordinate (r) at (3,0);
    \coordinate (d) at (0,-2);
    
    \node (11) {$R\modl\modl$};
    \node (12) at ($(11)+(r)$) {$\ab{M|_{\sigma}}$};

    \node (21) at ($(11) + (d)$){$S\modl\modl$};
    \node (22) at ($(21)+(r)$) {$\ab{M}$};

    \node (3) at ($(22) + (d)$) {$\modr S$};
    
    \draw[->,thick] (11) to (12);

    \draw[->,thick] (21) to (22);

    \draw[->,thick] (21) to node[left,yshift=-0.5em]{$\codefect$}(3);
    \draw[->,thick] (22) to node[right]{$\simeq$}(3);
    \draw[->,thick] (11) to node[left]{$\sigma_{\ast}$} (21);
    \draw[->,thick] (12) to node[right]{$\sigma_{\ast}$} (22);
  \end{tikzpicture}
\end{center}
By \Cref{example:length_detects_being_zero}, the invariant $K_0(\modr S) \xrightarrow{\length_S} \Z$ detects being zero.
Thus, we have
\[\length_{\sigma}( \mathcal{G} ) = \length_S(\codefect( \sigma_{\ast}( \mathcal{G} ) )) = 0\]
if and only if
\[\codefect( \sigma_{\ast}( \mathcal{G} ) ) \cong 0.\]
Since $\sigma_{\ast}: \ab{M_{|\sigma}} \rightarrow \ab{M}$ is a faithful functor, this in turn is equivalent to
\[
\asBeh{\mathcal{G}}{M_{|\sigma}} \cong 0
\]
which (by Paragraph \Cref{para:def_ab}) is equivalent to
\[
\mathcal{G}({M_{|\sigma}}) \cong 0.
\]
The claim follows.
\end{proof}

\section{The rank of a finitely presented functor over an integral domain}\label{section:the_rank}

By an \textbf{integral domain} we mean a \emph{commutative} ring whose zero ideal is prime.
We define the rank of a finitely presented functor over an integral domain as an additive invariant and show that it is compatible with the classical notion of the rank of a module by means of the left and right module classes that we introduced in \Cref{definition:lmc_rmc}. As a consequence, we will see that the rank of the co- and contravariant defects are the same (see \Cref{corollary:defect_codefect_same_rank}).

\begin{para}
For a field $Q$, the dimension is an additive invariant for the abelian category $Q\vecl$ which gives an isomorphism $K_0( Q\vecl ) \cong \Z$.
\end{para}

\begin{para}
Let $R$ be an integral domain and let $Q := \Quot(R)$ denote its field of fractions. 
For $M \in R\modl$, we set
\[
\rank( M ) := \dim_Q( Q \otimes_R M ) \in \Z.
\]
If $R$ is coherent, then $R\modl$ is an abelian category, and the rank defines an additive invariant for $R\modl$.
\end{para}

\begin{definition}
Let $R$ be an integral domain and let $Q := \Quot(R)$ denote its field of fractions. 
Evaluation at $Q$ induces an additive invariant
\[
\rank := K_0( \ev{Q} ): K_0(R\modl\modl) \rightarrow K_0(Q\vecl) \cong \Z
\]
which we call the \textbf{rank} of $\mathcal{G} \in R\modl\modl$, i.e., we have
\[
\rank( \mathcal{G} ) = \dim_{Q}( \mathcal{G}( Q ) ).
\]
\end{definition}

\begin{remark}\label{remark:rank_detects_being_zero}
We have 
\[
\mathcal{G}( Q ) = 0 \hspace{1em}\text{if and only if}\hspace{1em} \rank( \mathcal{G} ) = 0,
\]
i.e., the rank detects being zero at $Q$.
\end{remark}

\begin{lemma}\label{lemma:rank_lcm_rcm}
Let $R$ be an integral domain and let $\mathcal{G} \in R\modl\modl$.
Then
\[
\rank( \mathcal{G} ) = \rank( \lmc( \mathcal{G} ) ) = \rank( \rmc( \mathcal{G} ) ).
\]
\end{lemma}
\begin{proof}
It suffices to test the equality $\rank( \mathcal{G} ) = \rank( \lmc( \mathcal{G} ) )$ for $\mathcal{G} = \Hom(M,-)$, where $M \in R\modl$.
Thus, we need to see that
\[
\dim_Q(\Hom( M, Q ) ) = \dim_Q( Q \otimes M ),
\]
which follows directly from the isomorphism
\[
\Hom_R(M,Q) \cong \Hom_Q( (Q \otimes_R M), Q).
\]

For the second equation, it suffices to test the equality $\rank( \mathcal{G} ) = \rank( \rmc( \mathcal{G} ) )$ for $\mathcal{G} = (N \otimes -)$, where $N \in \modr R$. Unraveling the definitions, this amounts to the trivial equality
\[
\dim_Q( N \otimes Q ) = \dim_Q( N \otimes Q ). \qedhere
\]
\end{proof}

\begin{corollary}\label{corollary:defect_codefect_same_rank}
Let $R$ be a coherent integral domain. Then we have for all finitely presented functors $\mathcal{G}$ in $R\modl\modl$:
\[
\rank( \mathcal{G} ) = \rank( \defect( \mathcal{G} ) ) = \rank( \codefect( \mathcal{G} ) ).
\]
\end{corollary}
\begin{proof}
Combine \Cref{lemma:rank_lcm_rcm} and \Cref{lemma:lmc_defect_rcm_codefect}.
\end{proof}

\section{Applications to finitely presented functors over Dedekind domains}\label{section:applications}

In this section, we discuss the problem of how to compute the vanishing locus of a finitely presented functor over a Dedekind domain and provide a solution in \Cref{theorem:main_Dedekind}.
The case of Dedekind domains has two advantages: the points of the Ziegler spectrum are explicitly known (see \Cref{example:zg_dedekind}) and we can work with the structure theorem of finitely generated modules.

\begin{remark}\label{remark:structure_theorem_Dedekind}
We recall the structure theorem for finitely generated (=finitely presented) modules over Dedekind domains (see, e.g., \cite[Chapter II.4]{FroehlichTaylor}).
Let $R$ be a Dedekind domain and let $M \in R\modl$.
Then
\[
M \cong \pr(M) \oplus \bigoplus_{ \substack{\pid \in \mSpec(R) \\ l\geq 1} } (R/\pid^l)^{\oplus \mult_{\pid, l}^M},
\]
where $\pr(M) \in \Proj( R\modl )$ is a finitely presented projective $R$-module and $\mult_{\pid, l}^M$ are non-negative integers for $\pid \in \mSpec( R ), l \geq 1$, where only finitely many of them are non-zero.
The isomorphism class of $\pr(M)$ and the integers $\mult_{\pid, l}^M$ are uniquely determined by the isomorphism class of $M$.
Moreover, for every projective object $P \in \Proj( R\modl )$ of rank $n \geq 1$, we have an isomorphism
\[
P \cong R^{n-1} \oplus Q
\]
with $Q \in \Proj( R\modl )$ and $\rank( Q ) = 1$. The isomorphism class of $Q$ is uniquely determined by the isomorphism class of $P$.
\end{remark}

\subsection{Description of the Grothendieck group}

In this subsection, we describe the Grothendieck group of $R\modl\modl$ over a Dedekind domain $R$ in \Cref{corollary:K_0_Dedekind}.

\begin{lemma}\label{lemma:pr_mult}
Let $R$ be a Dedekind domain. The maps
\begin{align*}
\pr: \obj( |R\modl| ) &\rightarrow K_0( \Proj( R\modl ) ) \\
M &\mapsto [\pr(M)]
\end{align*}
and
\begin{align*}
\mult^{(-)}_{\pid,l}: \obj( |R\modl| ) &\rightarrow  \Z \\
M &\mapsto \mult^M_{\pid,l}
\end{align*}
define additive invariants for $|R\modl|$, where $\pr(M) \in \Proj( R\modl )$ and $\mult^M_{\pid,l}$ for $\pid \in \mSpec( R ), l \geq 1$ are the invariants of $M \in R\modl$ given by the structure theorem for finitely presented modules over Dedekind domains (see \Cref{remark:structure_theorem_Dedekind}).
\end{lemma}
\begin{proof}
Follows from the fact that the isomorphism class of $\pr(M)$ and the integers $\mult_{\pid, l}^M$ are uniquely determined by the isomorphism class of $M$.
\end{proof}

\begin{lemma}\label{lemma:K0_additive_Dedekind}
Let $R$ be a Dedekind domain.
The following map is an isomorphism:
\begin{align*}
K_0( |R\modl| ) &\xrightarrow{\cong} K_0( \Proj( R\modl ) ) \oplus \bigoplus_{ \substack{\pid \in \mSpec(R) \\ l\geq 1} } \Z \\
[M] &\mapsto ( [\pr(M)], (\mult^{M}_{\pid, l})_{\pid, l} )
\end{align*}
\end{lemma}
\begin{proof}
The map is well-defined by \Cref{lemma:pr_mult}.
Let $e_{\pid,l}$ denote the standard basis vector that is zero in each component except for the $(\pid,l)$-component.
We can define an inverse of the map in the statement via the assignments
\[
K_0(\Proj( R\modl ) ) \ni [P] \mapsto [P] \in K_0( |R\modl| )
\]
and
\[
e_{\pid,l} \mapsto [R/\pid^l] \in K_0( |R\modl| ). \qedhere
\]
\end{proof}

\begin{definition}\label{definition:pr_mult_for_fpfunctors}
Let $R$ be a Dedekind domain. For $\mathcal{G} \in R\modl\modl$, we define
\[
\pr( \mathcal{G} ) := \pr( \lmc( \mathcal{G} ) ) ) \in K_0( \Proj( R\modl ) )
\]
and the \textbf{multiplicities}
\[
\mult_{\pid,l}^{ \mathcal{G} } := \mult_{\pid,l}^{ \lmc(\mathcal{G}) } \in \Z.
\]
\end{definition}

\begin{para}
Let $R$ be a Dedekind domain.
We note that the conversion from left module classes to right module classes (see \Cref{corollary:K_0_of_left_and_right_modules}) maps the class of $R/\pid^l \in R\modl$ to the class of $R/\pid^l \in \modr R$.
This can be seen after tensoring with the localized ring $R_{\pid}$, which is a DVR.
It follows that we could also have used the right module class for defining the multiplicities in \Cref{definition:pr_mult_for_fpfunctors}.
However, a projective left module is converted to its dual right module (see \Cref{example:conversion_of_projectives}). In this sense, we made a choice when we defined $\pr( \mathcal{G} )$ in \Cref{definition:pr_mult_for_fpfunctors} via $\lmc$ and not via $\rmc$. However, this choice will not turn out to be relevant for our findings.
\end{para}

\begin{remark}
By \Cref{theorem:K_0_fpfunctors}, we can pull back additive invariants of $|R\modl|$ to additive invariants of $R\modl\modl$ via taking the left module classes.
In particular, \Cref{definition:pr_mult_for_fpfunctors} defines additive invariants for $R\modl\modl$.
\end{remark}

\begin{corollary}\label{corollary:K_0_Dedekind}
Let $R$ be a Dedekind domain.
The following map is an isomorphism:
\begin{align*}
K_0( R\modl\modl ) &\xlongrightarrow{\cong} K_0( \Proj( R\modl ) ) \oplus \bigoplus_{ \substack{\pid \in \mSpec(R) \\ l\geq 1} } \Z \\
[\mathcal{G}]
&\longmapsto \left( \pr(\mathcal{G}), (\mult^{\mathcal{G}}_{\pid,l})_{\pid,l} \right)
\end{align*}
\end{corollary}
\begin{proof}
Follows from \Cref{theorem:K_0_fpfunctors} and \Cref{lemma:K0_additive_Dedekind}.
\end{proof}

\subsection{Description of \texorpdfstring{$\ab{M}$}{ab(M)} for indecomposable pure-injective modules}

In this subsection, we describe the categories $\ab{M}$ for all indecomposable pure-injective modules $M$ over a Dedekind domain.

\begin{theorem}\label{theorem:abM_for_pure_injectives_Dedekind}
We use the notation from \Cref{example:zg_dedekind}.
Let $R$ be a Dedekind domain.
Then we have the following equivalences of categories:
\begin{align*}
    \ab{F_{\pid, n}} &\simeq \modr R/\pid^n \\
    \ab{A_{\pid}} &\simeq \modr R_{\pid} \\
    \ab{P_{\pid}} &\simeq (R_{\pid}\modl)^{\op} \\
    \ab{Q} &\simeq Q\vecl
\end{align*}
\end{theorem}
\begin{proof}
We will use Paragraphs \Cref{para:defect_equiv}, \Cref{para:codefect_equiv}, and \Cref{para:description_of_sigma_ast} for the concrete determination of the categories in the statement.
Let $\sigma: R \rightarrow R/\pid^n$ be the canonical ring epimorphism. 
We set $M := R/\pid^n \in R/\pid^n\Modl$ and note that $F_{\pid, n} \cong M|_{\sigma}$.
We note that $M$ is a faithfully flat $R/\pid^n$-module (tensoring with $M$ is the identity functor on $R/\pid^n\Modl$).
It follows that
\begin{align*}
    \ab{F_{\pid, n}} \simeq \ab{M} \simeq \modr R/\pid^n.
\end{align*}

Let $\sigma: R \rightarrow R_{\pid}$ be the canonical localization map, which is a ring epimorphism. 
We set $M := \varprojlim_n{R_{\pid}/\pid_{\pid}^n} \in R_{\pid}\Modl$ and note that $A_{\pid} \cong M|_{\sigma}$.
We note that $M$ is a faithfully flat $R_{\pid}$-module, since it is the completion of a noetherian local ring (see, e.g., \cite[Theorem 8.14]{Matsumura}).
It follows that
\begin{align*}
    \ab{A_{\pid}} \simeq \ab{M} \simeq \modr R_{\pid}.
\end{align*}

Let $\sigma: R \rightarrow R_{\pid}$ again be the canonical localization map.
We set $M := E( R_{\pid}/\pid_{\pid} ) \in R_{\pid}\Modl$ and note that $P_{\pid} \cong M|_{\sigma}$ (see, e.g., \cite[Theorem 18.14 (vi)]{Matsumura}).
We note that $M$ is an injective $R_{\pid}$-module by construction.
Moreover, it is a cogenerator of $R_{\pid}\Modl$ (see, e.g., \cite[Theorem 19.10]{Lam}).
It follows that
\begin{align*}
    \ab{P_{\pid}} \simeq \ab{M} \simeq (R\modl)^{\op}.
\end{align*}

Let $\sigma: R \rightarrow Q$ be the canonical localization map, which is a ring epimorphism. 
We set $M := Q \in Q\Modl$ and note that $Q \cong M|_{\sigma}$.
It is trivial that $Q$ is a faithfully flat $Q$-vector space.
It follows that
\[
    \ab{Q} \simeq \ab{M} \simeq \modr Q \simeq Q\vecl. \qedhere
\]
\end{proof}

\begin{example}\label{example:no_additive_invariant_that_detects_zero}
Let $R$ be a Dedekind domain and $\pid \in \mSpec( R )$.
Then there is no additive invariant that detects being zero at $P_{\pid}$.
Indeed, by \Cref{theorem:abM_for_pure_injectives_Dedekind}, we have
\[
\ab{P_{\pid}} \simeq (R_{\pid}\modl)^{\op}.
\]
But $K_0( (R_{\pid}\modl)^{\op} ) \cong \Z$, where the isomorphism is given by the rank of an $R_{\pid}$-module, and there are non-zero modules whose rank is $0$.
By \Cref{corollary:abM_detecs_zero_iff_there_exists}, there is no additive invariant that detects being zero at $P_{\pid}$.
Similarly, there is no additive invariant that detects being zero at $A_{\pid}$.
We note, however, that there is an additive invariant for $R_{\pid}\modl\modl$ that detects being zero at $P_{\pid} \oplus A_{\pid}$, see \Cref{corollary:AplusP_detection}.
\end{example}

\subsection{Hilbert functions}\label{subsection:hilbert_function}

In this subsection, we introduce the Hilbert function of a finitely presented functor $\mathcal{G}$ over a Dedekind domain. It is an additive invariant that allows us to understand the vanishing of $\mathcal{G}$ on all points of the form $F_{\pid,n} \in \Zg(R)$, for $\pid \in \mSpec(R)$, $n \geq 1$.

\begin{definition}
    Let $R$ be a Dedekind domain and let $\mathcal{G} \in R\modl\modl$.
    The \textbf{Hilbert function} of $\mathcal{G}$ is defined as
     \begin{align*}
     \mSpec( R ) \times \{ n \in \Z \mid n \geq 1 \} &\xrightarrow{\Hilb( \mathcal{G} )} \Z \\
     (\pid, n) &\mapsto \length_{\sigma_{\pid,n}}( \mathcal{G} ),
     \end{align*}
     where $\sigma_{\pid,n}: R \rightarrow R/\pid^n$ denotes the canonical ring epimorphism for $\pid \in \mSpec(R)$, $n \geq 1$.
\end{definition}

\begin{remark}
The function $\mathcal{G} \mapsto \Hilb( \mathcal{G} )$ is an additive invariant for $R\modl\modl$. Here, we regard $\Hilb( \mathcal{G} )$ as an element of the abelian group of functions  of type 
\[\mSpec( R ) \times \{ n \in \Z \mid n \geq 1 \} \rightarrow \Z.\]
\end{remark}

\begin{lemma}\label{lemma:hilb_detects_all_fpn}
Let $R$ be a Dedekind domain, $\pid \in \mSpec( R )$, $n \geq 1$. Then
\[
\text{$\Hilb( \mathcal{G} )(\pid, n) = 0$ if and only if $\mathcal{G}(F_{\pid,n}) = 0$.}
\]
\end{lemma}
\begin{proof}
Follows from \Cref{theorem:K_0_zero_detection}.
\end{proof}

\begin{notation}
We introduce functions of type $\mSpec( R ) \times \{ n \in \Z \mid n \geq 1 \} \rightarrow \Z$:
\[
\lambda: (\qid, n) \mapsto n
\]
and
\[
\mu_{\pid, l}: (\qid, n) \mapsto \delta_{\pid, \qid} \cdot \min\{l,n\},
\]
for $\pid \in \mSpec(R)$, $l \geq 1$, where $\delta_{\pid, \qid}$ denotes the Kronecker delta.
\end{notation}

\begin{lemma}\label{lemma:Hilb_of:P}
Let $R$ be a Dedekind domain and let $P \in \Proj(R\modl)$ be a projective module of rank $1$. Then
\[
\Hilb( \Hom( P, - ) ) = \lambda.
\]
\end{lemma}
\begin{proof}
Let $\sigma: R \rightarrow R/\pid^n$ denote the canonical ring epimorphism for $\pid \in \mSpec(R)$ and $n \geq 1$.
Using \Cref{lemma:compute_sigma_ast}, we compute
\begin{align*}
    \codefect( \sigma_{\ast}( \Hom_R( P, - ) ) ) &\cong \codefect( \Hom_{R/\pid^n}( R/\pid^n \otimes_R P, - ) ) \\
    &\cong \codefect( \Hom_{R/\pid^n}( R/\pid^n , - ) ) \\
    &\cong \codefect( R/\pid^n \otimes_{R/\pid^n} - )  \cong R/\pid^n.
\end{align*}
Since $\length_{R/\pid^n}( R/\pid^n ) = n$, the claim follows.
\end{proof}

\begin{lemma}\label{lemma:Hilb_of:Rmodpn}
Let $R$ be a Dedekind domain, $\pid \in \mSpec(R)$, $l \geq 1$.
Then
\[
\Hilb( \Hom( R/\pid^l, - ) ) = \mu_{\pid, l}.
\]
\end{lemma}
\begin{proof}
Let $\sigma: R \rightarrow R/\qid^n$ denote the canonical ring epimorphism for $\qid \in \mSpec(R)$ and $n \geq 1$.
Using \Cref{lemma:compute_sigma_ast}, we compute
\begin{align*}
    \codefect( \sigma_{\ast}( \Hom_R( R/\pid^l, - ) ) ) &\cong \codefect( \Hom_{R/\qid^n}( R/\qid^n \otimes_R R/\pid^l, - ) )
\end{align*}
If $\pid \neq \qid$, then $R/\qid^n \otimes_R R/\pid^l \cong 0$.
If $\pid = \qid$, then $R/\qid^n \otimes_R R/\pid^l \cong R/\qid^{\min\{l,n\}}$, and we compute
\begin{align*}
\codefect( \Hom_{R/\qid^n}( R/\qid^n \otimes_R R/\pid^l, - ) ) &\cong \codefect( \Hom_{R/\qid^n}( R/\qid^{\min\{l,n\}}, - ) ) \\
&\cong \Hom_{R/\qid^n}( R/\qid^{\min\{l,n\}}, R/\qid^n ) \cong R/\qid^{\min\{l,n\}}
\end{align*}
which is of length $\min\{l,n\}$. The claim follows.
\end{proof}

\begin{theorem}\label{theorem:compute_hilbert_function}
Let $R$ be a Dedekind domain.
The Hilbert function of a finitely presented functor $\mathcal{G} \in R\modl\modl$ is given by the following formula:
\[
\Hilb( \mathcal{G} ) = \rank( \mathcal{G} ) \cdot \lambda + \sum_{\substack{\pid \in \mSpec(R) \\ l \geq 1} }\mult_{\pid,l}^{ \mathcal{G} } \cdot \mu_{\pid, l}.
\]
\end{theorem}
\begin{proof}
From \Cref{corollary:K_0_Dedekind}, we get
\[
[ \mathcal{G} ] = \sum_{i \in I} a_i \cdot [\Hom(P_i,-)] + \sum_{\substack{\pid \in \mSpec(R) \\ l \geq 1} }\mult_{\pid,l}^{ \mathcal{G} } \cdot [ \Hom( R/\pid^l, - ) ],
\]
where $I$ is a finite index set, the $P_i$ are projective modules of rank $1$, $a_i \in \Z$.
Evaluation at $Q$ yields
\begin{align*}
\rank( \mathcal{G} ) = \sum_{i \in I}a_i.
\end{align*}
Now, the claim follows from \Cref{lemma:Hilb_of:P} and \Cref{lemma:Hilb_of:Rmodpn}.
\end{proof}

\begin{definition}
Let $R$ be a Dedekind domain and let $\mathcal{G} \in R\modl\modl$.
The \textbf{Hilbert polynomial} of $\mathcal{G}$ at $\pid \in \mSpec( R )$ is defined as
\begin{align*}
\HilbP( \mathcal{G} )(\pid, x) := \rank( \mathcal{G} ) \cdot x + \sum_{l \geq 1}\mult_{\pid,l}^{ \mathcal{G} } \cdot l \in \Z[x].
\end{align*}
\end{definition}

\begin{para}
The Hilbert polynomial is an additive invariant of $R\modl\modl$ for each $\pid \in \mSpec(R)$. By \Cref{theorem:compute_hilbert_function}, it coincides with the Hilbert function for $n$ large enough, more precisely, for each $\pid \in \mSpec(R)$ we have 
\[
\HilbP( \mathcal{G} )(\pid, n) = \Hilb( \mathcal{G} )( \pid, n )
\]
for $n \geq \max(\{l \mid \mult_{\pid,l}^{ \mathcal{G} } \neq 0\} \cup \{1\})$, 
\end{para}

\subsection{Computing the vanishing locus of a finitely presented functor}

The following theorem summarizes our results for the computation of the vanishing locus of a finitely presented functor over a Dedekind domain.

\begin{theorem}\label{theorem:main_Dedekind}
We use the notation of \Cref{example:zg_dedekind}.
Let $R$ be a Dedekind domain and let $\mathcal{G} \in R\modl\modl$.
Then
\[
\V( \mathcal{G} ) = \V_{\Hilb}( \mathcal{G} ) \cup \V_{\defect}( \mathcal{G} ) \cup \V_{\codefect}( \mathcal{G} ) \cup \V_{\rank}( \mathcal{G} ),
\]
where the sets on the right hand side are given as follows:
\begin{align*}
 \V_{\Hilb}( \mathcal{G} ) &:= \{ F_{\pid,n} \mid \text{$\pid \in \mSpec(R)$, $n \geq 1$, such that $\Hilb( \mathcal{G} )(\pid,n) = 0$} \}\\
 \V_{\defect}( \mathcal{G} ) &:= \{ P_{\pid} \mid \text{$\pid \in \mSpec(R)$ such that $R_{\pid} \otimes \defect( \mathcal{G} ) \cong 0$} \}\\
 \V_{\codefect}( \mathcal{G} ) & := \{ A_{\pid} \mid \text{$\pid \in \mSpec(R)$ such that $\codefect(\mathcal{G}) \otimes R_{\pid}$} \cong 0\}\\
 \V_{\rank}( \mathcal{G} ) & := 
 \begin{cases}
    \{ Q \} & \text{$\rank( \mathcal{G} ) = 0$} \\
    \emptyset & \text{otherwise}
  \end{cases}
\end{align*}
\end{theorem}
\begin{proof}
By \Cref{example:zg_dedekind}, we need to check when $\mathcal{G}$ vanishes on the following points:
\[
\Zg(R) = \{ F_{\pid,n}, P_{\pid}, A_{\pid}, Q \mid  n \geq 1, \pid \in \mSpec( R ) \}.
\]
By \Cref{lemma:hilb_detects_all_fpn}, $\mathcal{G}(F_{\pid,n}) \cong 0$ can be checked via the vanishing of the Hilbert function.
By \Cref{theorem:checking_kernels_of_defects}, $\mathcal{G}(P_{\pid}) \cong 0$ can be checked via $R_{\pid} \otimes \defect( \mathcal{G} ) \cong 0$.
Again by \Cref{theorem:checking_kernels_of_defects}, $\mathcal{G}(A_{\pid}) \cong 0$ can be checked via $\codefect(\mathcal{G}) \otimes R_{\pid} \cong 0$.
By \Cref{remark:rank_detects_being_zero}, the rank detects being zero at $Q$. This completes the proof.
\end{proof}

\begin{remark}
The vanishing locus $\V_{\Hilb}( \mathcal{G} )$ of the Hilbert function can be computed explicitly as follows.
First, we observe that the set 
\[
\{ \pid \in \mSpec(R) \mid \text{$\exists l \geq 1$ such that $\mult_{\pid,l}^{ \mathcal{G} } \neq 0$} \}
\]
is finite by \Cref{corollary:K_0_Dedekind}. We set $m_{\pid} := \max(\{l \mid \mult_{\pid,l}^{ \mathcal{G} } \neq 0\} \cup \{1\})$ for each $\pid \in \mSpec(R)$.
Second, we define
\[
\V_{\HilbP}( \mathcal{G} ) := \{ F_{\pid,n} \mid \text{$\pid \in \mSpec(R)$, $n \geq 1$, such that $\HilbP( \mathcal{G} )(\pid,n) = 0$}\}.
\]
Now, we can compute $\V_{\Hilb}( \mathcal{G} )$ as the following union:
\[
\{ F_{\pid,n} \in \V_{\Hilb}( \mathcal{G} ) \mid \text{$\pid \in \mSpec(R)$, $n < m_{\pid}$} \} \cup \{ F_{\pid,n} \in \V_{\HilbP}( \mathcal{G} ) \mid \text{$\pid \in \mSpec(R)$, $n \geq m_{\pid}$} \}.
\]
The set on the left hand side involves the testing of only finitely many primes and integers. The set on the right hand side are simply the vanishing loci of linear functions.
\end{remark}

\begin{para}
For the following examples, we use the following notation:
Let $R$ be a Dedekind domain at let $\pid \in \mSpec(R)$.
Let $p \in \pid_{\pid} \subseteq R_{\pid}$ be a generator of the principal ideal of the localization of $R$ at $\pid$. Note that $p$ acts on $R/\pid^n$ for every $n \geq 1$, since $R/\pid^n \cong R_{\pid}/\pid_{\pid}^n$ as $R$-modules.
\end{para}

\begin{example}\label{example:S_1}
We have the following short exact sequence of $R$-modules:
\begin{center}
  \begin{tikzpicture}[mylabel/.style={fill=white}]
      \coordinate (r) at (4,0);
      \node (A) {$0$};
      \node (B) at ($(A)+0.5*(r)$) {$R/\pid$};
      \node (C) at ($(B) + (r)$) {$R/\pid^{2}$};
      \node (D) at ($(C) + (r)$) {$R/\pid$};
      \node (E) at ($(D) + 0.5*(r)$) {$0$};
      \draw[->,thick] (A) to (B);
      \draw[->,thick] (B) to node[above]{$(p)$}(C);
      \draw[->,thick] (C) to node[above]{$(1)$}(D);
      \draw[->,thick] (D) to (E);
  \end{tikzpicture}
\end{center}
We define $\mathcal{S}_{\pid,1}$ via the following exact sequence in $R\modl\modl$:
\begin{center}
  \begin{tikzpicture}[mylabel/.style={fill=white}]
      \coordinate (r) at (4,0);
      \node (As) {$0$};
      \node (A) at ($(A)+0.35*(r)$){$\mathcal{S}_{\pid,1}$};
      \node (B) at ($(A)+ 0.75*(r)$) {$(R/\pid \otimes -)$};
      \node (C) at ($(B) + (r)$) {$(R/\pid^{2}  \otimes -)$};
      \node (D) at ($(C) + (r)$) {$(R/\pid \otimes -)$};
      \node (E) at ($(D) + 0.5*(r)$) {$0$};
      \draw[->,thick] (As) to (A);
      \draw[->,thick] (A) to (B);
      \draw[->,thick] (B) to node[above]{$(p \otimes -)$}(C);
      \draw[->,thick] (C) to node[above]{$(1 \otimes -)$}(D);
      \draw[->,thick] (D) to (E);
  \end{tikzpicture}
\end{center}
Then we can compute:
\begin{itemize}
    \item $\mathcal{S}_{\pid,1}(R/\pid) \cong R/\pid$,
    \item $\mathcal{S}_{\pid,1}(R/\pid^m) \cong 0$ for $m \neq 1$,
    \item $\mathcal{S}_{\qid,1}(R/\qid^m) \cong 0$ for $\qid \neq \pid$, $m \geq 1$,
    \item $\mathcal{S}_{\pid,1}(P) \cong 0$, for $P \in \Proj( R\modl )$.
\end{itemize}
In particular, $\mathcal{S}_{\pid,1}$ is a simple object in $R\modl\modl$. We have
\begin{itemize}
    \item $\Hilb( \mathcal{S}_{\pid,1} ) = 2\cdot\mu_{\pid,1} - \mu_{\pid,2}$,
    \item $\defect( \mathcal{S}_{\pid,1} ) \cong 0$,
    \item $\codefect( \mathcal{S}_{\pid,1} ) \cong 0$,
    \item $\rank( \mathcal{S}_{\pid,1} ) = 0$.
\end{itemize}
Thus, \Cref{theorem:main_Dedekind} yields $\V( \mathcal{S}_{\pid,1} ) = \Zg(R) \setminus \{ F_{\pid,1} \}$.
\end{example}

\begin{example}\label{example:S_2}
Let $n \geq 2$.
We have the following short exact sequence of $R$-modules:
\begin{center}
  \begin{tikzpicture}[mylabel/.style={fill=white}]
      \coordinate (r) at (4,0);
      
      \node (A) {$0$};
      \node (B) at ($(A)+0.5*(r)$) {$R/\pid^n$};
      \node (C) at ($(B) + (r)$) {$R/\pid^{n-1} \oplus R/\pid^{n+1}$};
      \node (D) at ($(C) + (r)$) {$R/\pid^n$};
      \node (E) at ($(D) + 0.5*(r)$) {$0$};
      
      \draw[->,thick] (A) to (B);
      \draw[->,thick] (B) to node[above]{$\pmatrow{1}{p}$}(C);
      \draw[->,thick] (C) to node[above]{$\pmatcol{-p}{1}$}(D);
      \draw[->,thick] (D) to (E);
  \end{tikzpicture}
\end{center}
We define $\mathcal{S}_{\pid,n}$ via the following exact sequence in $R\modl\modl$:
\begin{center}
  \begin{tikzpicture}[mylabel/.style={fill=white}]
      \coordinate (r) at (4,0);
      
      \node (As) {$0$};
      \node (A) at ($(As)+0.35*(r)$){$\mathcal{S}_{\pid,n}$};
      \node (B) at ($(A)+0.5*(r)$) {$(R/\pid^n \otimes -)$};
      \node (C) at ($(B) + 1.25*(r)$) {$((R/\pid^{n-1} \oplus R/\pid^{n+1}) \otimes -)$};
      \node (D) at ($(C) + 1.25*(r)$) {$(R/\pid^n \otimes -)$};
      \node (E) at ($(D) + 0.5*(r)$) {$0$};
      
      \draw[->,thick] (As) to (A);
      \draw[->,thick] (A) to (B);
      \draw[->,thick] (B) to node[above, scale = 0.7]{$(\pmatrow{1}{p} \otimes -)$}(C);
      \draw[->,thick] (C) to node[above, scale = 0.7]{$(\pmatcol{-p}{1} \otimes -)$}(D);
      \draw[->,thick] (D) to (E);
  \end{tikzpicture}
\end{center}
Then we can compute:
\begin{itemize}
    \item $\mathcal{S}_{\pid,n}(R/\pid^n) \cong R/\pid$,
    \item $\mathcal{S}_{\pid,n}(R/\pid^m) \cong 0$ for $m \neq n$,
    \item $\mathcal{S}_{\qid,n}(R/\qid^m) \cong 0$ for $\qid \neq \pid$, $m \geq 1$,
    \item $\mathcal{S}_{\pid,n}(P) \cong 0$, for $P \in \Proj( R\modl )$.
\end{itemize}
In particular, $\mathcal{S}_{\pid,n}$ is a simple object in $R\modl\modl$. We have
\begin{itemize}
    \item $\Hilb( \mathcal{S}_{\pid,n} ) = 2\cdot\mu_{\pid,n} - \mu_{\pid,n-1} - \mu_{\pid,n+1}$,
    \item $\defect( \mathcal{S}_{\pid,n} ) \cong 0$,
    \item $\codefect( \mathcal{S}_{\pid,n} ) \cong 0$,
    \item $\rank( \mathcal{S}_{\pid,n} ) = 0$.
\end{itemize}
Thus, \Cref{theorem:main_Dedekind} yields $\V( \mathcal{S}_{\pid,n} ) = \Zg(R) \setminus \{ F_{\pid,n} \}$.
\end{example}

\begin{corollary}\label{corollary:detection_intersection_defect_codefect}
Let $R$ be a Dedekind domain and let $\mathcal{G} \in R\modl\modl$.
The following statements are equivalent:
\begin{enumerate}
    \item $\mathcal{G}$ is a finite extension of simple functors of the form $\mathcal{S}_{\pid,n}$ for $\pid \in \mSpec(R)$, $n \geq 1$,
    \item $\defect( \mathcal{G} ) \cong 0$ and $\codefect( \mathcal{G} ) \cong 0$,
    \item $\HilbP( \mathcal{G} ) = 0$.
\end{enumerate}
\end{corollary}
\begin{proof}
Being a finite extension of simple functors of the described form means that $\mathcal{G}$ lies in the Serre subcategory spanned by those simple functors: 
\[
\mathbf{I} := \langle \mathcal{S}_{\pid,n} \mid \pid \in \mSpec(R), n \geq 1 \rangle \in \Serre( R\modl\modl ).
\]
We set
\[
X := \{ A_{\pid}, P_{\pid}, Q \mid \pid \in \mSpec(R) \} \subseteq \Zg(R).
\]
We claim that the following equality holds:
\[
\mathbf{I}(X) = \mathbf{I}.
\]
Since $X$ is a closed set, this equality is equivalent to
\[
X = \V( \mathbf{I} ),
\]
which is true by \Cref{example:S_1} and \Cref{example:S_2}:
\[
\V( \mathbf{I} ) = \bigcap_{\pid,n} \V( \mathcal{S}_{\pid,n} ) = \bigcap_{\pid,n} \Zg(R) \setminus \{ F_{\pid,n} \} = \Zg(R) \setminus ( \bigcup_{\pid,n}\{F_{\pid,n}\} ) = X.
\]
Thus, $\mathcal{G} \in \mathbf{I}$ if and only if $\mathcal{G}$ vanishes on all points in $X$. Now, the equivalence of statement $1$ and $2$ follows from \Cref{theorem:main_Dedekind}.

Now, if $\mathcal{G} \in \mathbf{I}$, then its Hilbert function is zero except for finitely many pairs of primes and integers. Thus, its Hilbert polynomial vanishes.

Conversely, assume that $\HilbP( \mathcal{G} ) = 0$. Then $\rank( \mathcal{G} ) = 0$. It follows from \Cref{theorem:compute_hilbert_function} that
\[
\Hilb( \mathcal{G} ) =  \sum_{\substack{\pid \in \mSpec(R) \\ l \geq 1} }\mult_{\pid,l}^{ \mathcal{G} } \cdot \mu_{\pid, l}.
\]
Note that in the sum on the right hand side, only finitely many primes are involved.
It follows that $\mathcal{G}( R/\pid^n ) \not\cong 0$ for only finitely many pairs of primes $\pid$ and integers $n \geq 1$. Let $\pid$ be one such prime. Then $\pid \mathcal{G}$ is a proper subfunctor of $\mathcal{G}$, and $\mathcal{G}/\pid \mathcal{G}$ is a direct sum of simple functors of the form $\mathcal{S}_{\pid,n}$. Inductively, we get that $\mathcal{G} \in \mathbf{I}$.
\end{proof}

\begin{corollary}\label{corollary:AplusP_detection}
Let $R$ be a DVR. Then the Hilbert polynomial detects being zero at $P \oplus A$, where we use the notation of \Cref{example:topology_localpid}.
\end{corollary}
\begin{proof}
In the case where $R$ is a DVR, the module $P$ is an injective cogenerator for $R\Modl$ and $A$ is a faithfully flat $R$-module. Thus, 
\[\kernel( \ev{P \oplus A} ) = \kernel( \defect ) \cap \kernel( \codefect ).\]
Now, the claim follows from \Cref{corollary:detection_intersection_defect_codefect}.
\end{proof}

\bibliographystyle{alpha}
\bibliography{biblio}

\end{document}

%% file: header.tex
\usepackage[utf8]{inputenc} 
\usepackage[T1]{fontenc}
\usepackage{a4wide}
\usepackage{lmodern}
\usepackage[english]{babel}
\usepackage{mathrsfs}

\usepackage{mathtools}
\usepackage{latexsym}
\usepackage{amssymb}
\usepackage{amsthm}
\usepackage{amsmath}
\usepackage{caption}
\usepackage{mathabx}

\usepackage{colortbl}
\usepackage{tensor}

\usepackage[all]{xy}
\usepackage{verbatim}
\usepackage{listings}
\usepackage{fancyvrb}

\usepackage{graphicx}

\usepackage[dvipsnames]{xcolor}
\usepackage{accents} 
\usepackage{enumerate}
\usepackage{wrapfig}
\usepackage{tikz}
\usepackage{tikz-cd}
\usetikzlibrary{automata,shapes,arrows,matrix,backgrounds,positioning,plotmarks,calc,patterns,matrix,decorations.pathreplacing,decorations.pathmorphing,decorations.text,decorations.markings}
\usepackage[colorinlistoftodos,shadow]{todonotes}

\usepackage{multirow}
\usepackage{mdwlist}

\usepackage{stmaryrd}
\usepackage{mathdots} 

\usepackage{fancyvrb}

\usepackage[toc,page]{appendix}
\usepackage{float}

\usepackage{extarrows}
\usepackage{pdflscape}
\usepackage{rotating}

\usepackage[
colorlinks=true,
linkcolor=black, 
anchorcolor=black,
citecolor=black, 
urlcolor=black, 
]{hyperref}
\usepackage{cleveref} 

\newtheoremstyle{mytheoremstyle} 
    {5pt}                    
    {5pt}                    
    {\itshape}                   
    {\parindent}                           
    {\bf}                   
    {.}                          
    {.5em}                       
    {}  

\theoremstyle{mytheoremstyle}

\newtheorem{theorem}{Theorem}[section]

\newtheorem{lemma}[theorem]{Lemma}

\newtheorem{corollary}[theorem]{Corollary}

\newtheoremstyle{mytdefintionstyle} 
    {5pt}                    
    {5pt}                    
    {\rm}                   
    {\parindent}                           
    {\bf}                   
    {.}                          
    {.5em}                       
    {}  

\theoremstyle{remark}
\newtheorem{remark}[theorem]{Remark}

\theoremstyle{mytdefintionstyle}
\newtheorem{definition}[theorem]{Definition}

\newtheorem{example}[theorem]{Example}

\newtheorem{convention}[theorem]{Convention}
\newtheorem{notation}[theorem]{Notation}

\newtheorem{para}[theorem]{}

\newtheoremstyle{exmp_contd} 
{\topsep} {\topsep}%
{\upshape}
{}
{\bfseries}
{}
{ }
{\thmname{#1}\,\thmnumber{ #2}\thmnote{#3}\enspace(continued)}

\theoremstyle{exmp_contd}

\usepackage{xspace}
\input{pre.tex}

\tikzset{round left paren/.style={ncbar=0.5cm,out=120,in=-120}}
\tikzset{round right paren/.style={ncbar=0.5cm,out=60,in=-60}}
\newcolumntype{C}[1]{>{\centering\arraybackslash$}p{#1}<{$}}
\newlength{\mycolwd}
\usepackage{array, xcolor}
\definecolor{lightgray}{gray}{0.8}
\newcolumntype{L}{>{\raggedleft}p{0.28\textwidth}}
\newcolumntype{R}{p{0.8\textwidth}}

\definecolor{ctcolor}{gray}{0.95}

\definecolor{ctucolor}{gray}{0.85}

\makeatletter
\newcommand{\thickhline}{%
    \noalign {\ifnum 0=`}\fi \hrule height 1pt
    \futurelet \reserved@a \@xhline
}
\newcolumntype{"}{@{\hskip\tabcolsep\vrule width 1pt\hskip\tabcolsep}}
\makeatother

\usepackage{enumitem}
\newlist{theoremenumerate}{enumerate}{1}
\setlist[theoremenumerate]{label=(\arabic{theoremenumeratei}), ref=\thetheorem.(\arabic{theoremenumeratei}),noitemsep}

%% file: pre.tex
\DeclareMathOperator{\Hom}{\mathrm{Hom}}
\DeclareMathOperator{\Ext}{\mathrm{Ext}}
\DeclareMathOperator{\kernel}{\mathrm{ker}}
\DeclareMathOperator{\rank}{\mathrm{rank}}
\DeclareMathOperator{\cokernel}{\mathrm{cok}}

\newcommand{\op}{\mathrm{op}}

\newcommand{\Tr}{\mathrm{Tr}}
\newcommand{\Kb}{\mathrm{K^b}}
\newcommand{\Db}{\mathrm{D^b}}

\newcommand{\vecl}{\text{-}\mathbf{vec}}

\newcommand{\Frm}{\mathbf{Frm}}
\newcommand{\Top}{\mathbf{Top}}

\newcommand{\Modl}{\text{-}\mathbf{Mod}}
\newcommand{\Modr}{\mathbf{Mod}\text{-}}
\newcommand{\modl}{\text{-}\mathbf{mod}}
\newcommand{\modr}{\mathbf{mod}\text{-}}
\newcommand{\Ab}{\mathbf{Ab}}

\newcommand{\ab}[1]{\mathbf{ab}({#1})}

\newcommand{\DAGJ}{D_{AGJ}}
\newcommand{\unit}{\mathrm{unit}}
\newcommand{\Irr}{\mathrm{Irr}}
\newcommand{\lmc}{\mathrm{lmc}}
\newcommand{\rmc}{\mathrm{rmc}}
\newcommand{\length}{\mathrm{length}}
\newcommand{\Hilb}{\mathrm{Hilb}}
\newcommand{\HilbP}{\mathrm{HilbP}}

\newcommand{\defect}{\mathrm{Defect}}
\newcommand{\Zg}{\mathrm{Zg}}
\newcommand{\Serre}{\mathrm{Serre}}
\newcommand{\codefect}{\mathrm{Covdefect}}
\newcommand{\ev}[1]{\mathrm{Eval}_{#1}}

\newcommand{\asBeh}[2]{{#1}[{#2}]}

\newcommand{\ringmap}{\sigma}

\newcommand{\Z}{\mathbb{Z}}

\newcommand{\AC}{\mathbf{A}}
\newcommand{\BC}{\mathbf{B}}

\newcommand{\CC}{\mathbf{C}}

\newcommand{\TC}{\mathbf{T}}
\newcommand{\ff}{\mathcal{F}}

\newcommand{\M}{M}

\newcommand{\Spec}{\mathrm{Spec}}

\newcommand{\pmatrow}[2]{ \begin{pmatrix}{#1} & {#2} \end{pmatrix} }

\newcommand{\pmatcol}[2]{ \begin{pmatrix}{#1} \\ {#2} \end{pmatrix} }

\newcommand{\Supp}{\mathrm{Supp}}
\newcommand{\Ann}{\mathrm{Ann}}
\newcommand{\pr}{\mathrm{pr}}
\newcommand{\Proj}{\mathrm{Proj}}
\newcommand{\Inj}{\mathrm{Inj}}
\newcommand{\mult}{\mathrm{mult}}

\newcommand{\obj}{\mathrm{obj}}
\newcommand{\V}{\mathbb{V}}

\newcommand{\pid}{\mathfrak{p}}
\DeclareMathOperator{\mSpec}{mSpec}

\DeclareMathOperator{\Quot}{Quot}

\newcommand{\qid}{\mathfrak{q}}